\documentclass[final,times, 12pt]{elsarticle}%

\usepackage{amsthm, amsfonts}
\usepackage{color}
\usepackage{graphicx}
\usepackage{subfigure}
\usepackage{amssymb}
\usepackage{amsmath,amsxtra}
\usepackage{multirow}
\usepackage{enumerate}
\usepackage{epstopdf}
\usepackage{stmaryrd,txfonts}

\newcommand{\ujump}{{\kappa_1}}
\newcommand{\ujn}{{\kappa_2}}

\newtheorem{theorem}{Theorem}[section]
\newtheorem{remark}[theorem]{Remark}

\newcounter{mnote}
\let\oldmarginpar\marginpar
\renewcommand\marginpar[1]{\-\oldmarginpar[\raggedleft\footnotesize #1]
  {\raggedright\footnotesize #1}}

\usepackage{footnote}
\usepackage{booktabs}

\usepackage[notcite,notref,color]{showkeys}
\usepackage{rotating}

\numberwithin{equation}{section}

\usepackage[colorlinks,linkcolor=black,
anchorcolor=black,
citecolor=black]{hyperref}
\usepackage[shortlabels]{enumitem}
\setlist[enumerate]{nosep}
\usepackage{color, soul}
\soulregister\cite7
\usepackage{setspace}

\usepackage[mathscr]{eucal}

\def\rot{{\rm rot}}
\def\curl{{\rm curl}}

\def\dv{{\rm div}}
\usepackage[margin=1.2in]{geometry}  

\journal{arXiv}

\begin{document}

\title{A Natural Deep Ritz Method for Essential Boundary Value Problems}

\author{Haijun Yu}
\ead{hyu@lsec.cc.ac.cn}

\author{Shuo Zhang\corref{cor1}}
\cortext[cor1]{Corresponding author.}
\ead{szhang@lsec.cc.ac.cn}

\address{LSEC, Institute of Computational Mathematics and
    Scientific/Engineering Computing, Academy of Mathematics and System
    Sciences, Chinese Academy of Sciences, Beijing 100190; University of
    Chinese Academy of Sciences, Beijing, 100049; People's Republic of China}

\begin{abstract}
Deep neural network approaches show promise in solving partial differential equations. However, unlike traditional numerical methods, they face challenges in enforcing essential boundary conditions. The widely adopted penalty-type methods, for example, offer a straightforward implementation but introduces additional complexity due to the need for hyper-parameter tuning; moreover, the use of a large penalty parameter can lead to artificial extra stiffness, complicating the optimization process. In this paper, we propose a novel, intrinsic approach to impose essential boundary conditions through a framework inspired by intrinsic structures. We demonstrate the effectiveness of this approach using the deep Ritz method applied to Poisson problems, with the potential for extension to more general equations and other deep learning techniques. Numerical results are provided to substantiate the efficiency and robustness of the proposed method.
\end{abstract}


\maketitle


%
%
%
\section{Introduction}

%

In recent years, there has been a rapidly growing interest in using deep neural networks (DNNs) to solve partial differential equations (PDEs). Early attempts to apply neural networks to differential equations date back over three decades, with Hopfield neural networks \cite{hopfield_neural_1982} being employed to represent discretized solutions \cite{lee_neural_1990}. Soon after, methodologies were developed to construct closed-form numerical solutions using neural networks \cite{milligen_neural_1995}. Since then, extensive research has focused on solving differential equations with various types of neural networks, including feedforward neural networks \cite{lagaris_artificial_1998,mcfall_artificial_2009,lagaris_neural_2000, mcfall_automated_2013}, radial basis networks \cite{maiduy_numerical_2001}, and wavelet networks \cite{li_integration_2010}. With the advancement of deep learning techniques \cite{hinton_fast_2006, kingmaAdamMethodStochastic2015, he_deep_2016}, neural networks with substantially more hidden layers have become powerful tools. Innovations such as rectified linear unit (ReLU) functions \cite{glorotDeepSparseRectifier2011}, generative adversarial networks (GANs) \cite{goodfellow_generative_2014}, and residual networks (ResNets) \cite{he_deep_2016} exemplify these advances, showcasing the strong representational capabilities of DNNs \cite{petersen_neural_2018,li_better_2020,li_PowerNet_2020, tang_ChebNet_2024, he_ReLU_2020,shen_deep_2020}. These developments have spurred the creation of numerous DNN-based methods for PDEs, including the deep Galerkin method (DGM) \cite{sirignano_dgm_2018}, deep Ritz method (DRM) \cite{e_deepRitz_2018}, physics-informed neural networks (PINNs) \cite{raissi_pinn_2019}, finite neuron method (FNM) \cite{xu_finite_2020}, weak adversarial networks (WANs) \cite{zang_weak_2020}, and mixed residual methods (MIM) \cite{lyu_mim_2022}. These methods have been widely adopted across various applications, successfully addressing complex problems modeled by differential equations \cite{e_deepRitz_2018,liao_deepNitsche_2021,raissi_hidden_2020,berg_unified_2018,long_PDEnet_2019,khoo_solving_2021,yu_onsagernet_2021a,chen_Constructing_2024}.

\medskip

In the design and implementation of neural network-based methods, the imposition of boundary conditions is a critical challenge. Notably, this issue is also encountered in certain classical numerical methods, such as finite element methods, where handling boundary conditions can be complex enough to require techniques like Nitsche's method \cite{nitsche_1971}, later refined by Stenberg \cite{stenberg_techniques_1995}. However, the challenges differ significantly in neural network-based approaches. Unlike classical numerical methods, which leverage basis functions or discretization stencils with compact supports or sparse structures, neural network methods utilize DNNs as trial functions, which are globally defined. Consequently, enforcing boundary conditions, even for problems that are straightforward in classical methods, becomes nontrivial due to the global structure of DNNs. For the natural boundary conditions, the deep Ritz method reformulates the original problem into a variational form, which can reduce the smoothness requirements and potentially lower the training cost by allowing natural boundary conditions to be imposed without additional operations. However, because the trial functions within the approximation sets are generally non-interpolatory, imposing essential boundary conditions remains a challenging task.

\medskip

To date, three primary approaches have been developed for addressing essential boundary conditions in deep learning-based numerical methods. The first approach is the conforming method, which aims to construct neural network functions that exactly satisfy the essential boundary conditions \cite{sheng_PFNN_2021,berg_unified_2018,lyu_mim_2022}. Generally, the network function $u_{NN}(x)$ is represented as the combination of two parts: $\displaystyle u_{NN}(x) = u_b(x) + d_\Gamma(x) u^0_{NN}(x)$, one reflecting the essential boundary condition, and the other vanishing on the boundary $\Gamma$ by the aid of a ``distance function" or a ``geometry-aware" function $d_\Gamma(x)$. Both test and trial functions can be constructed this way. However, when the domain has a complicated boundary (or even not that complicated), it is not easy to construct a distance function to preserve the asymptotic equivalence.

Another one is the penalty method, which is a very general concept and belongs to the so-called nonconforming method \cite{e_deepRitz_2018,raissi_pinn_2019,sirignano_dgm_2018,zhang_learning_2020,zang_weak_2020,huang_augmented_2022}. For this method, an additional surface term is introduced into the variational formulation to enforce the boundary conditions. Take the Poisson equation with Dirichlet boundary condition \eqref{eq:PoissonStrong} as example:
\begin{equation}\label{eq:PoissonStrong}
    \left\{
    \begin{array}{rll}
        -\Delta u & =f & \mbox{in}\,\Omega,
        \\
        u         & =g & \mbox{on}\,\Gamma=\partial\Omega.
    \end{array}
    \right.
\end{equation}
The deep Ritz method~\cite{e_deepRitz_2018} minimize the following objective
\begin{equation} \label{eq:drm}
    \mathcal{L}_{DRM}(u) = \biggl[ \sum_{x_j \in \mathcal{D}} \frac12 |\nabla u(x_j) |^2 - f(x_j) u(x_j) \biggr] + \beta \sum_{x_j \in \mathcal{D}_\Gamma}\bigl(u(x_j) - u_b(x_j)\bigr)^2,
\end{equation}
where $\mathcal{D}$ and $\mathcal{D}_\Gamma$ define the training data set in the domain and on the boundary, respectively. PINN method is a least square method for the strong form of the PDE, but the the handling of the essential boundary condition is similar to deep Ritz method:
\begin{equation} \label{eq:pinn}
    \mathcal{L}_{PINN}(u) = \biggl[ \sum_{x_j \in \mathcal{D}} |\Delta u(x_j) + f(x_j) |^2 \biggr] + \beta \sum_{x_j \in \mathcal{D}_\Gamma}\bigl(u(x_j) - u_b(x_j)\bigr)^2.
\end{equation}
Careful balancing of terms within the functional framework is essential to ensure the well-posedness and accuracy of the scheme. Addressing this issue, the deep Nitsche method, as proposed in \cite{liao_deepNitsche_2021}, applies Nitsche's variational formula to second-order elliptic problems to avoid the use of a large penalty parameter. Nevertheless, some degree of tuning remains necessary for the penalty parameter, and a theoretical basis for determining an optimal penalty value is still absent.

In contrast to the penalty method, the Lagrange multiplier method addresses essential boundary conditions by treating them as constraints within the minimization process. This method has been effectively used to impose essential boundary conditions in finite element methods \cite{babuska_finite_1973} and wavelet methods \cite{dahmen_appending_2001}. When the approximation function spaces are appropriately chosen satisfying the so-called inf-sup condition, this method can achieve optimal convergence rates \cite{babuska_finite_1973,dahmen_appending_2001}. While the Lagrange multiplier method can also enforce boundary conditions in neural network-based methods, its effectiveness depends on the stable construction and efficient resolution of the extra constrained optimization problem.

\medskip

In this paper, we introduce a novel neural network-based method for solving essential boundary value problems. Our approach involves transforming the original problem into a sequence of natural boundary value problems, which are then solved sequentially or concurrently using the deep Ritz method. Unlike the previously mentioned approaches, this technique constructs a new framework for imposing essential boundary conditions. We refer to this method as the natural deep Ritz method. This approach simplifies the training process and avoids introducing additional errors associated with boundary condition enforcement. To validate our method, we examine essential boundary and interface value problems for second-order divergence-form equations with constant, variable, or discontinuous coefficients, providing numerical examples that demonstrate the effectiveness.

Evidently, a primary ingredient of the proposed method lies in its adjoint approach to handling essential boundary conditions. This approach is grounded in the mathematical framework of the de Rham complex and its dual complex, which serve as foundational structures. By leveraging these complexes, which connect kernel spaces to specific range spaces, we can represent the difference between the solutions of natural and essential boundary value problems as the solution to another natural boundary value problem. This formulation allows us to construct a purely natural approach equivalent to the original problem.

While we do not delve extensively into the formal structure of the de Rham and dual complexes, it is important to highlight that our method diverges from the traditional mixed formulations common in classical numerical methods. Notably, we do not introduce the gradient of the unknown function as an auxiliary variable. Moreover, unlike classical mixed formulations, our approach avoids the need for constructing a saddle point problem, which would typically require rigorous continuous and discrete inf-sup conditions for stability and accuracy. In our framework, the solution is reduced to solving three elliptic subproblems using a standard machine learning algorithm. This approach eliminates the need for training an additional network to capture the boundary representation, tuning penalty parameters, or ensuring inf-sup conditions for a boundary Lagrangian multiplier. The conciseness of the present method is among its most significant advantages, both in theory and implementation.

\medskip

The remaining parts of the paper are organized as follows. In Section \ref{sec:natural}, we present the equivalent natural boundary value problem formulation of the respective essential boundary value problems. In Section \ref{sec:ndrm}, the deep Ritz method based on the natural formulation, namely the natural deep Ritz methods, is given. Numerical experiments are presented in Section \ref{sec:ne} to verify the proposed method. We end the paper with some concluding remarks in Section \ref{sec:conc}

\section{A natural formulation of the essential boundary value problems}
\label{sec:natural}

In this section, we derive natural formulations for the second-order problems of divergence form with constant, variable, and discontinuous coefficients, respectively; i.e., we rewrite the essential boundary value problems and interface value problems to a series of natural boundary value problems and interface value problems to solve. We are focused on Laplace problems on two-dimensional domains here, and the method can be generated to higher dimensions, as well as to other self-adjoint problems. 

In this paper, $\Omega\subset\mathbb{R}^2$ stands for a simply connected domain with a boundary $\Gamma$, and we use $L^2(\Omega)$, $H^1(\Omega)$, $H^1_0(\Omega)$, $H^{-1}(\Omega)$, $H^{1/2}(\Gamma)$ and $H^{-1/2}(\Gamma)$ for the standard Sobolev spaces. 

\subsection{Poisson equations of Dirichlet type}

We first consider the model problem with constant coefficients: \eqref{eq:PoissonStrong}.
Its variational formulation is to find $u\in H^1_g(\Omega):=\bigl\{\, w\in
    H^1(\Omega):w|_\Gamma=g\,\bigr\}$, such that
\begin{equation}\label{eq:modelpoisson}
    (\nabla u,\nabla v)=\langle f,v\rangle_{H^{-1}(\Omega)\times H^1_0(\Omega)},\ \ \forall\,v\in H^1_0(\Omega).
\end{equation}
Here, $\langle \cdot,\cdot\rangle_{H^{-1}(\Omega)\times H^1_0(\Omega)}$ stands for the duality between $H^{-1}(\Omega)$ and $H^1_0(\Omega)$. In the sequel, we use $\langle\cdot,\cdot\rangle$ to denote dualities of different kinds, while the subscripts may be dropped when no ambiguity is introduced.

\begin{theorem}\label{thm:poisson} Let $u$ be the solution of
    \eqref{eq:modelpoisson}, and $u^*$ be obtained by the four steps below:
    \begin{enumerate}
        \item Find $ \tilde u \in H^1_\Gamma(\Omega):=\{w\in H^1(\Omega):\int_\Gamma w=0\}$, such that
              \begin{equation}\label{eq:poissonneumann}
                  (\nabla  \tilde u ,\nabla v)=\langle \tilde f,v\rangle_{(H^1_\Gamma(\Omega))'\times H^1_\Gamma(\Omega)},\ \ \forall\,v\in H^1_\Gamma(\Omega),
              \end{equation}
              where $\tilde f$ is any extension of $f$ onto
              $(H^1_\Gamma(\Omega))'$ such that $\langle\tilde
                  f,v\rangle=\langle f,v\rangle$ for $v\in H^1_0(\Omega)$.
        \item Find a $\varphi\in H^1(\Omega)$, such that
              \begin{equation}\label{eq:poissonkernelmodi}
                  (\curl\varphi,\curl\psi)=\langle\partial_t (g- \tilde u |_\Gamma),\psi\rangle_\Gamma,\ \forall\,\psi\in H^1(\Omega);
              \end{equation}
Here, the scalar curl operator is defined as $\curl\,w(x,y):= (\partial_{y} w, - \partial_{x} w)$, and $\langle\cdot,\cdot\rangle_\Gamma$ is a duality between $H^{-1/2}(\Gamma)$ and $H^{1/2}(\Gamma)$, which evaluates as the $L^2$ inner product on $\Gamma$ for sufficiently smooth functions. 
        \item  Find a $u_c\in H^1(\Omega)$, such that
              \begin{equation}\label{eq:poissonopera}
                  (\nabla u_c,\nabla v)=(\nabla  \tilde u {-} \curl\varphi,\nabla v),\ \forall\,v\in H^1(\Omega).
              \end{equation}
        \item  Set $u^*=u_c-C$, with $C=\frac{1}{|\gamma|}\int_{\gamma} (u_c-g)$
              for any $\gamma\subset \Gamma$ such that $|\gamma|\neq 0$.
    \end{enumerate}
    Then $u^*=u$.
\end{theorem}

\begin{proof}
    By \eqref{eq:modelpoisson} and \eqref{eq:poissonneumann}, $(\nabla u-\nabla
        \tilde u ,\nabla v)=0,\ \ \forall\,v\in H^1_0(\Omega)$ and it follows that
    $\nabla u-\nabla \tilde u =\curl\varphi$ for some $\varphi\in H^1(\Omega)$.
    Further, $\rot\,\curl\varphi=0$. Therefore, for any $\psi\in H^1(\Omega)$,
    we have
    $(\curl\varphi,\curl\psi)=(\rot\curl\varphi,\psi)+\langle\curl\varphi\cdot\mathbf{t},\psi\rangle_{\Gamma}=\langle(\nabla
        u-\nabla \tilde u
        )\cdot\mathbf{t},\psi\rangle_{\Gamma}=\langle\partial_\mathbf{t}(g- \tilde u
        |_\Gamma),\psi\rangle_\Gamma$, namely $\varphi$ satisfied
    \eqref{eq:poissonkernelmodi}. Now we obtain by \eqref{eq:poissonopera} that
    $\nabla u_c=\nabla u$. Then $u_c-u$ is a constant which can be corrected by
    Step (4) and finally, we are lead to that $u^*=u$. The proof is completed.
\end{proof}

\begin{remark}
    \begin{enumerate}
        \item The solutions of the second and third steps are not unique up to
              constant, though, these solutions will give the same correct solution
              at the end of the algorithm.
        \item To obtain $\tilde u$, we may solve for $\tilde u\in H^1(\Omega)$
              $$
                  (\nabla \tilde u,\nabla v) =\langle \tilde f,v-\fint_\Gamma v\rangle_{(H^1_\Gamma(\Omega))'\times H^1_\Gamma(\Omega)},
                  \ \ \forall\,v\in H^1(\Omega);
              $$
        \item The last step can be done by least square.
    \end{enumerate}
\end{remark}

\begin{remark}
    We can interprate formally the first three subproblems in the formulation of natural boundary value problems as below:
    \begin{enumerate}
        \item The boundary value problem corresponding to
              \eqref{eq:poissonneumann}:
              \begin{equation}
                \label{eq:NDB1}
                \left\{
                  \begin{array}{rll}
                      -\Delta  \tilde u                             & =f                         & \mbox{in}\,\Omega,
                      \\
                      \frac{\partial \tilde u }{\partial\mathbf{n}} & =
                      {-} \frac{1}{|\Gamma|}\int_\Omega f,  & \mbox{on}\,\partial\Omega.
                  \end{array}
                  \right.
              \end{equation}
        \item The boundary value problem corresponding to
              \eqref{eq:poissonkernelmodi}:
              \begin{equation}
                \label{eq:NDB2}
                \left\{
                  \begin{array}{rll}
                      -\Delta \varphi             & =0                                              & \mbox{in}\,\Omega,
                      \\
                      \curl\varphi\cdot\mathbf{t} & =\partial_\mathbf{t}g-\partial_\mathbf{t} \tilde u , & \mbox{on}\,\partial\Omega.
                  \end{array}
                  \right.
              \end{equation}
        \item  The boundary value problem corresponding to
              \eqref{eq:poissonopera}:
              \begin{equation}
                \label{eq:NDB3}
                \left\{
                  \begin{array}{rll}
                      -\Delta u_c                             & =f                                                                  & \mbox{in}\,\Omega,
                      \\
                      \frac{\partial u_c}{\partial\mathbf{n}} & =\partial_\mathbf{n} \tilde u {-} \partial_\mathbf{t}\varphi, & \mbox{on}\,\partial\Omega.
                  \end{array}
                  \right.
              \end{equation}
    \end{enumerate}

\end{remark}

\subsection{Elliptic problem with varying coefficient in divergence form}
Let $\mathcal{A} $ be a varying coefficient matrix  such that
\begin{equation}\label{equation}
    \lambda |\xi|^2 \le \mathcal{A}_{ij}(x)\xi_i \xi_j \le \Lambda |\xi|^2\qquad\mbox{on}\ \ \Omega.
\end{equation}
We further consider a second order problem of divergence form:
\begin{equation}\label{eq:varc-strongform}
    \left\{
    \begin{array}{rll}
        -\dv(\mathcal{A} ^2\nabla u) & =f & \mbox{in}\,\Omega,
        \\
        u                       & =g & \mbox{on}\,\Gamma.
    \end{array}
    \right.
\end{equation}

It is useful to rewrite $-\dv\circ(\mathcal{A} ^2\nabla) =(-\dv\mathcal{A} )\circ(\mathcal{A}
    \nabla)$. Note that, equipped with proper spaces, the operators $-\dv\mathcal{A} $
and $\mathcal{A} \nabla$ are adjoint operators of each other, and we write the
variational formulation to be: find $u\in H^1_g(\Omega)$, such that
\begin{equation}\label{eq:modeldivformvc}
    (\mathcal{A} \nabla u,\mathcal{A} \nabla v)=\langle f,v\rangle,\ \ \forall\,v\in H^1_0(\Omega).
\end{equation}

\begin{theorem}
    Let $u$ be the solution of \eqref{eq:modeldivformvc}, and $u^*$ be obtained
    by the four steps below:
    \begin{enumerate}
        \item Find $ \tilde u \in H^1_\Gamma(\Omega):=\bigl\{w\in H^1(\Omega):\int_\Gamma w=0\bigr\}$, such that
              \begin{equation}\label{eq:divformvcneumann}
                  (\mathcal{A} \nabla  \tilde u ,\mathcal{A} \nabla v)=\langle \tilde f,v\rangle_{H^{-1}(\Omega)\times H^1_0(\Omega)},
                  \ \ \forall\,v\in H^1_\Gamma(\Omega);
              \end{equation}
        \item Find a $\varphi\in H^1(\Omega)$, such that
              \begin{equation}\label{eq:divformvckernelmodi}
                  ({\mathcal{A}^{-1} }\curl\varphi,{\mathcal{A}^{-1} }\curl\psi)
                  =\langle\partial_t (g- \tilde u |_\Gamma),\psi\rangle_\Gamma,\ \forall\,\psi\in H^1(\Omega);
              \end{equation}
        \item  Find a $u_c\in H^1(\Omega)$, such that
              \begin{equation}\label{eq:divformvcopera}
                  (\mathcal{A} \nabla u_c,\mathcal{A} \nabla v)
                  =(\mathcal{A} \nabla  \tilde u - \mathcal{A} ^{-1}\curl\varphi,\mathcal{A} \nabla v),
                  \ \forall\,v\in H^1(\Omega);
              \end{equation}
        \item  Set $u^*=u_c-C$, with $C=\frac{1}{|\gamma|}\int_{\gamma} (u_c-g)$
              for any $\gamma\subset \Gamma$ such that $|\gamma|\neq 0$.
    \end{enumerate}
    Then $u^*=u$.
\end{theorem}
\begin{proof}
    Note that the null space of $\dv\circ\mathcal{A} $ coincides with the range of
    $\mathcal{A}^{-1}\circ\curl$ equipped with proper spaces, and the proof is
    the same as that of Theorem \ref{thm:poisson}.
\end{proof}

\subsection{A simple approach for the interface problem}

We now consider the case that $\mathcal{A} $ is discontinuous. Let $\Gamma_0$ be an
interface that separates
$\overline\Omega=\overline\Omega_1\cup\overline\Omega_2$ with
$\mathring{\Omega_1}\cap\mathring{\Omega_2}=\emptyset$; see Figure
\ref{fig:interface} for an illustration.  We use $\mathbf{n}_i$ and
$\mathbf{t}_i$ for the outer unit normal vector and the corresponding unit
tangential vector for $\partial\Omega_i$, $i=1,2$.
\begin{figure}[htbp]
    \centering
    \includegraphics[width=0.6\textwidth]{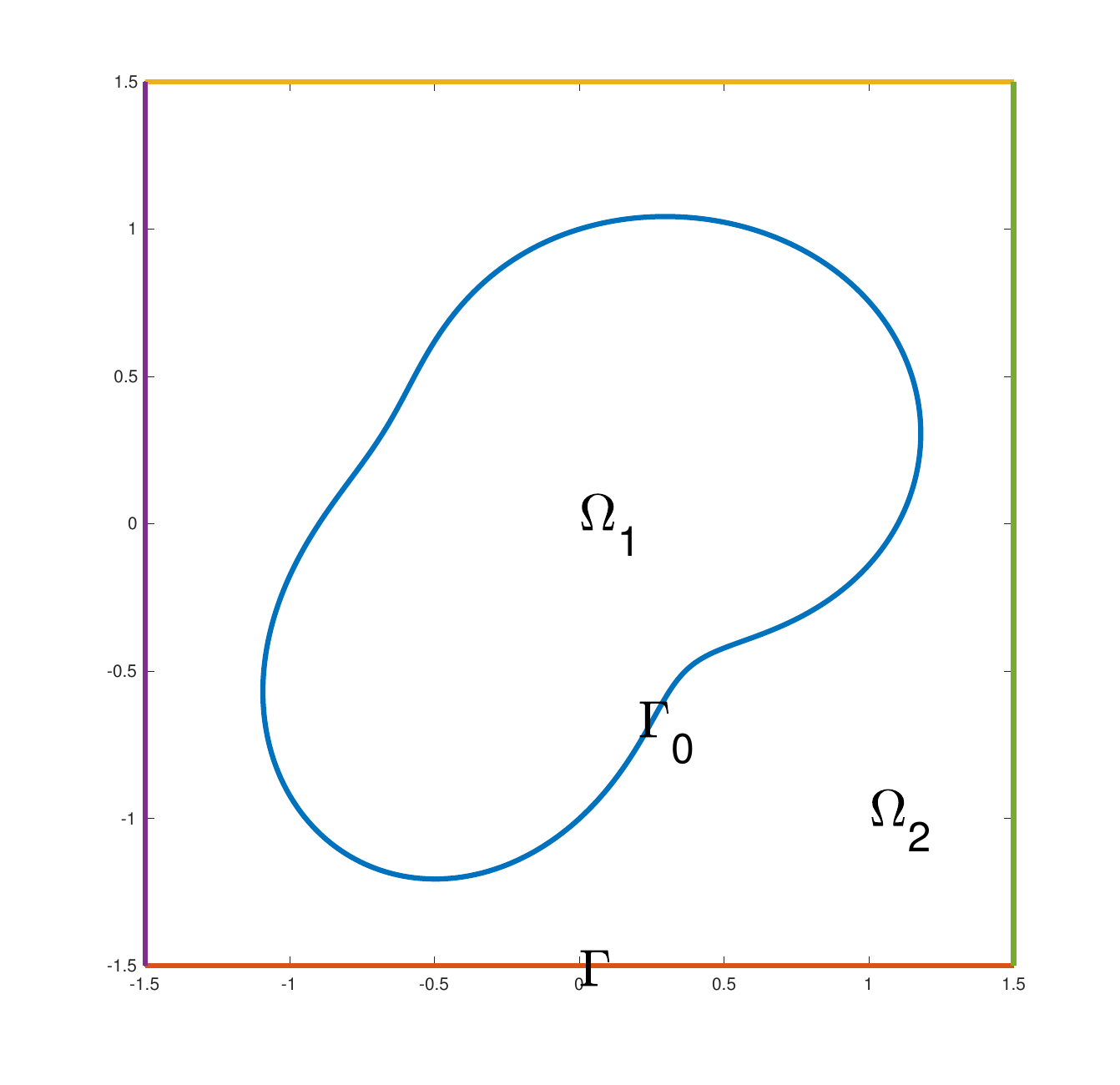}
    \caption{Illustration of the domain and the interface}\label{fig:interface}
\end{figure}

Assume $\mathcal{A} $ to be discontinuous across $\Gamma_0$. We consider the interface problem below:

\begin{equation}
    \left\{
    \begin{array}{rll}
        -\dv\mathcal{A} ^2\nabla u       & =f      & \mbox{in}\,\Omega_1\cup\Omega_2,
        \\
        u                           & =g      & \mbox{on}\,\Gamma,
        \\
        (\mathcal{A} ^2\nabla u)|_{\Omega_1}\cdot\mathbf{n}_1+(\mathcal{A} ^2\nabla u)|_{\Omega_2}\cdot\mathbf{n}_2
                                    & =\ujn   & \mbox{on}\,\Gamma_0,
        \\
        u|_{\Omega_1}-u|_{\Omega_2} & =\ujump & \mbox{on}\,\Gamma_0.
    \end{array}
    \right.
\end{equation}
The variational formulation is to find $u\in H^1(\Omega_1)\times
    H^1(\Omega_2):=\{w\in L^2(\Omega):w|_{\Omega_i}\in H^1(\Omega_i),\ i=1,2\}$,
such that
\begin{equation}\label{eq:modelinterfacial}
    (\mathcal{A} ^2\nabla u,\nabla v)_{\Omega_1\cup\Omega_2}
    =\langle f,v\rangle_{H^{-1}(\Omega)\times H^1_0(\Omega)}+\langle\ujn,v\rangle_{\Gamma_0},\ \ \forall\,v\in H^1_0(\Omega),
\end{equation}
and
\begin{equation} \label{eq:modelinterfacial-bc}
    u|_\Gamma=g,\ \ \mbox{and}\ \ \   u|_{\Omega_1}-u|_{\Omega_2} =\ujump\ \ \text{on } \Gamma_0.
\end{equation}
Here $\langle\cdot,\cdot\rangle_{\Gamma_0}$ is a duality between $H^{-1/2}(\Gamma_0)$ and $H^{1/2}(\Gamma_0)$, which evaluates as the $L^2$ inner product on $\Gamma_0$ for sufficiently smooth functions. 

\begin{theorem}\label{thm:interface} Let $u$ be the solution of
    \eqref{eq:modelinterfacial}-\eqref{eq:modelinterfacial-bc}, and $u^*$ be obtained by the four steps below:
    \begin{enumerate}
        \item Find a $ \tilde u \in H^1(\Omega)$, such that
              \begin{equation}\label{eq:interfacialneumann}
                  (\mathcal{A} \nabla  \tilde u ,\mathcal{A} \nabla v)_\Omega
                  =\langle \tilde f,v-\fint_\Gamma v\rangle+\langle \ujn, v-\fint_\Gamma v\rangle_{\Gamma_0},\ \ \forall\,v\in H^1(\Omega).
              \end{equation}
        \item Find a $\varphi\in H^1(\Omega)$, such that
              \begin{equation}\label{eq:interfacialkernelmodi}
                  ({\mathcal{A}^{-1} }\curl\varphi, {\mathcal{A}^{-1} }\curl\psi)_{\Omega_1\cup \Omega_2}=\langle \partial_{\mathbf{t}_1}\ujump,\psi\rangle_{\Gamma_0}+\langle \partial_\mathbf{t}(g- \tilde u |_{\Gamma}),\psi\rangle_\Gamma,\ \forall\,\psi\in H^1(\Omega).
              \end{equation}
        \item  Find a $u_c\in H^1(\Omega_1)\times H^1(\Omega_2)$, such that
              \begin{equation}\label{eq:interfacialopera}
                  (\mathcal{A} \nabla u_c,\mathcal{A} \nabla v)_{\Omega_i}=(\mathcal{A} \nabla  \tilde u - \mathcal{A} ^{-1}\curl\varphi,\mathcal{A} \nabla v)_{\Omega_i},\ \forall\,v\in H^1(\Omega_i), \, i=1,2.
              \end{equation}
        \item  Set $u^*|_{\Omega_1}=u_c|_{\Omega_1}-C_1$,
              $u^*|_{\Omega_2}=u_c|_{\Omega_2}-C_2$, with
              $C_2=\frac{1}{|\gamma|}\int_{\gamma} (u_c-g)$ for any $\gamma\subset
                  \Gamma$ such that $|\gamma|\neq 0$, and
              $C_1=\frac{1}{|\gamma_0|}\int_{\gamma_0} (u_c|_{\Omega_1} -
                  u^*|_{\Omega_2}-\kappa_1)$ for any $\gamma_0\subset \Gamma_0$ such that
              $|\gamma_0|\neq 0$.
    \end{enumerate}
    Then $u^*=u$.
\end{theorem}
\begin{proof}
    By the first item, $(\mathcal{A} ^2\nabla(u- \tilde u ),\nabla v)=0$ for $v\in
        H^1_0(\Omega)$, therefore, there exists $\varphi\in H^1(\Omega)$ such that
    $\mathcal{A} \nabla(u- \tilde u )=\mathcal{A} ^{-1}\curl \varphi$. It follows that
    $\rot\,\mathcal{A} ^{-2}\curl\varphi=0$. Then
    \begin{align*}
        ({\mathcal{A}^{-1} }\curl\varphi, {\mathcal{A}^{-1} }\curl\psi)_{\Omega_i}
         & =\langle \mathcal{A} ^{-2}\curl\varphi\cdot\mathbf{t}_i,\psi\rangle_{\partial\Omega_i}                                                                                                         \\
         & =\langle \nabla(u- \tilde u )\cdot\mathbf{t}_i,\psi\rangle_{\partial\Omega_i}=\langle \partial_{\mathbf{t}_i}(u|_{\partial\Omega_i}- \tilde u |_{\partial\Omega_i}),\psi\rangle_{\partial\Omega_i},
    \end{align*}
    for any $\psi\in H^1(\Omega)$, and further
    \begin{align*}
        ( {\mathcal{A}^{-1} }\curl\varphi, {\mathcal{A}^{-1} }\curl\psi)_{\Omega_1\cup \Omega_2}
         & =\sum_{i=1}^2\langle \partial_{\mathbf{t}_i}(u|_{\partial\Omega_i}- \tilde u |_{\partial\Omega_i}),\psi\rangle_{\partial\Omega_i}     \\
         & =\langle \partial_{\mathbf{t}_1}\ujump,\psi\rangle_{\Gamma_0}+\langle \partial_\mathbf{t}(g- \tilde u |_{\Gamma}),\psi\rangle_\Gamma.
    \end{align*}
    The assertion follows immediately.
\end{proof}

\begin{remark}
Again, it is helpful to understand the procedure by figuring out the respective strong forms related to equations
    \eqref{eq:interfacialneumann}-\eqref{eq:interfacialopera}.
    \begin{enumerate}
        \item By using integration by parts, we obtain the strong form of
              \eqref{eq:interfacialneumann}:
              \begin{equation}
                \label{eq:InterfaceS1}
                  \left\{
                  \begin{array}{rll}
                      -\nabla \cdot (\mathcal{A}^2 \nabla \tilde u)          & =f,                & \mbox{in}\,\Omega_1\cup \Omega_2,
                      \\
                      \mathbf{n}_1\cdot (\mathcal{A} ^2\nabla \tilde u)|_{\Omega_1}
                      + \mathbf{n}_2 \cdot (\mathcal{A} ^2\nabla \tilde u)|_{\Omega_2}
                                                                   & =\ujn,             & \mbox{on}\,\Gamma_0,
                      \\
                      \mathbf{n}\cdot (\mathcal{A}^2 \partial \tilde u )     & =
                      - \frac{1}{|\Gamma|} \left( \langle f, 1 \rangle_{\Omega}
                      + \langle k_2, 1 \rangle_{\Gamma_0} \right), & \mbox{on}\,\Gamma.
                  \end{array}
                  \right.
              \end{equation}
        \item The boundary value problem corresponding to
              \eqref{eq:interfacialkernelmodi} is:
              \begin{equation}
                \label{eq:InterfaceS2}
                  \left\{
                  \begin{array}{rll}
                      -\rot (\mathcal{A}^{-2} \curl\varphi) & =0                                 & \mbox{in}\,\Omega_1\cup \Omega_2,
                      \\
                      \mathbf{t}_1\cdot(\mathcal{A}^{-2}\curl\varphi)|_{\Omega_1} + \mathbf{t}_2\cdot(\mathcal{A}^{-2}\curl\varphi)|_{\Omega_2}
                                                               & =\partial_{\mathbf{t}_1} \kappa_1,
                                                               & \mbox{on}\,\Gamma_0,                                                   \\
                      \mathbf{t}\cdot(\mathcal{A}^{-2}\curl\varphi) & =\partial_\mathbf{t} (g- \tilde u),
                                                               & \mbox{on}\,\Gamma.
                  \end{array}
                  \right.
              \end{equation}
        \item  The boundary value problem corresponding to
              \eqref{eq:interfacialopera} is:
              \begin{equation}
                \label{eq:InterfaceS3}
                  \left\{
                  \begin{array}{rll}
                      -\nabla\cdot (\mathcal{A}^2 \nabla u_c)                                                 & =f,                          & \mbox{in}\,\Omega_i,
                      \\
                      \mathbf{n}_i \cdot (\mathcal{A}^2 \nabla u_c) - \mathbf{n}_i\cdot (\mathcal{A}^2 \nabla \tilde u) & =
                      - \partial_\mathbf{t}\varphi,                                                      & \mbox{on}\,\partial\Omega_i,
                  \end{array}
                  \right.
            \end{equation}
              for $i=1,2$.
    \end{enumerate}
    From these strong forms, we see clearly that $\tilde u$ and $\phi$ are
    both continuous functions but with derivative jumps on interface $\Gamma_0$. This softens the jumps between $u_c|_{\Omega_1}$ and $u_c|_{\Omega_2}$ across $\Gamma_0$ on both function value and derivative jumps.
\end{remark}

\section{Natural deep Ritz methods}
\label{sec:ndrm}

%

\subsection{Natural deep Ritz method for Poisson equations  with Dirichlet boundary conditions}
Note that the three equations
\eqref{eq:poissonneumann}-\eqref{eq:poissonkernelmodi}-\eqref{eq:poissonopera}
in weak forms correspond to three elliptic equations with Neumann boundary
conditions \eqref{eq:NDB1}-\eqref{eq:NDB2}-\eqref{eq:NDB3}, which can be efficiently solved using Deep Ritz method without boundary penalty. Details are given in the following three steps. As usual, we use $\Phi_{NN}(d,1)$ for the set of neural network functions outputting a 1-dim vector with a d-dim input vector.
\begin{enumerate}
  \item Find $u_1 \in \Phi_{N N} (d, 1) /\mathbb{R}$ by optimizing
  \begin{equation}
    \mathcal{L}_1 (u_1) : = \Bigl[ \sum_{\{ x_j, \omega_j \} \in \mathcal{D}}
    \frac{1}{2} | \nabla u_1 (x_j) |^2 \omega_j - f (x_j) (u_1 (x_j) - c_1)
    \omega_j \Bigr]
    + c_1^2, \label{eq:opt1}
  \end{equation}
  where $c_1 = \frac{1}{| \mathcal{D}_{\Gamma} |} \sum_{\{ x_j, \omega_j \}
  \in \mathcal{D}_{\Gamma}} u_1 (x_j) \omega_j$. $\mathcal{D}$ and
  $\mathcal{D}_{\Gamma}$ are the set of quadrature points and weights for
  domain $\Omega$ and its boundary $\Gamma$. Hereby, the term $c_1^2$ is added in the objective function to make the
  solution unique.

  \item Find $\varphi \in \Phi_{N N} (d, 1)$ by optimizing
  \begin{align}
    \mathcal{L}_2 (\varphi)
    := & \sum_{ \{ x_j, \omega_j\}\ \in \mathcal{D}} \frac{1}{2}
    [\curl \varphi (x_j)]^2  \omega_j  
     + \sum_{ \{ x_j, \omega_j\} \in \mathcal{D}_\Gamma
    } \bigl[ g (x_j) \partial_{\mathbf{\tau}} \varphi
    (x_j) + \partial_{\mathbf{\tau}} u_1 (x_j) \varphi (x_j) \bigr] \omega_j
    \nonumber \\
    &+ \Bigl[ \sum_{ \{x_j, \omega_j\} \in \mathcal{D}_\Gamma}
    \varphi(x_j)\omega_j \Bigr]^2. \label{eq:opt2}
  \end{align}
Again, the last term is added to make the solution unique.

  \item Find the solution $u_c\in \Phi_{N N} (d, 1)$ by minimizing:
  \begin{align}
    \mathcal{L}_3(u_c)
    :=
    & \sum_{ \{ x_j, \omega_j \} \in
    \mathcal{D}} \bigl| \nabla u_c (x_j) - \nabla u_1 (x_j) + \curl
    \varphi (x_j) \bigr|^2\omega_j \nonumber\\
    & + \Bigl[\sum_{ \{ x_j, \omega_j \} \in
    \mathcal{D}_\Gamma} ( u_c(x_j) - g(x_j) ) \omega_j \Bigr]^2.
    \label{eq:opt3}
  \end{align}
  The last term is a regularization term to make the integration of $u_c$
  and $g$ on boundary $\partial\Omega$ equal to each other. Then, $u_c$ is a
  proper numerical approximation of $u$.
\end{enumerate}

One may optimize the three equation \eqref{eq:opt1}-\eqref{eq:opt3} one by one,
or optimize $\mathcal{L}_1 + \mathcal{L}_2 +
\mathcal{L}_3$ all in one. To make the training procedure simpler, we take the latter approach in this paper.

The variable coefficient systems
\eqref{eq:interfacialneumann}-\eqref{eq:interfacialkernelmodi}-\eqref{eq:interfacialopera}
can be solved similarly by the proposed natural deep Ritz method. We omit the details to save space.

\subsection{Natural deep Ritz method for elliptic interface problems}

For inteface problem defined in
\eqref{eq:interfacialneumann}-\eqref{eq:interfacialkernelmodi}-\eqref{eq:interfacialopera},
it is more involved to design an efficient deep Ritz method.
We will use similar approach as in the Poisson equations cases to solve
\eqref{eq:interfacialneumann}-\eqref{eq:interfacialkernelmodi}, since both $u_1$
and $\varphi$ has no jump of function values on the interface $\Gamma_0$. We use two neural network functions to represent $u_c$, since it contains jumps on the
interface $\Gamma_0$.
We will
solve \eqref{eq:interfacialopera} with two neural networks (or one neural
network with two outputs  $\Phi_{NN}(d,2)$), one for each
subdomain $\Omega_{i}, i=1, 2$. The details are given below.
\begin{enumerate}
    \item Find $u_1 \in \Phi_{N N} (d, 1) /\mathbb{R}$ by optimizing
    \begin{align}
      \mathcal{L}_1 (u_1)
      : = & \Bigl[ \sum_{\{ x_j, \omega_j \} \in \mathcal{D}}
      \frac{1}{2} | \nabla u_1 (x_j) |^2 \mathcal{A}^2(x_j)\omega_j
      - f (x_j) (u_1 (x_j) - c_1) \omega_j  \Bigr] \\
      & - \Bigl[ \sum_{\{ x_j, \omega_j \} \in \mathcal{D}_{\Gamma_0}}
      \kappa_2 (x_j) (u_1 (x_j) - c_1) \omega_j \Bigr]
      + c_1^2, \label{eq:opt1-interface}
    \end{align}
    where $c_1 = \frac{1}{| \mathcal{D}_{\Gamma} |} \sum_{\{ x_j, \omega_j \}
    \in \mathcal{D}_{\Gamma}} u_1 (x_j) \omega_j$. $\mathcal{D}$,
    $\mathcal{D}_{\Gamma}$ and $\mathcal{D}_{\Gamma_0}$ are the set of quadrature points and weights for
    domain $\Omega$, boundary $\Gamma$ and interface $\Gamma_0$.

    \item Find $\varphi \in \Phi_{N N} (d, 1)$ by optimizing
    \begin{align}
      \mathcal{L}_2 (\varphi)
      := & \sum_{ \{ x_j, \omega_j\}\ \in \mathcal{D}}
      \frac{1}{2}
      [\curl \varphi (x_j)]^2 \mathcal{A}^{-2}(x_j) \omega_j  \nonumber\\
       & + \sum_{ \{ x_j, \omega_j\} \in \mathcal{D}_\Gamma
      } \bigl[ g (x_j) \partial_{\mathbf{\tau}} \varphi
      (x_j) + \partial_{\mathbf{\tau}} u_1 (x_j) \varphi (x_j) \bigr] \omega_j \nonumber\\
      & + \Bigl[ \sum_{ \{ x_j, \omega_j\} \in \mathcal{D}_{\Gamma_0}
      }  \kappa_1 (x_j) \partial_{\mathbf{\tau}} \varphi
      (x_j) \omega_j \Bigr]
      + \Bigl[ \sum_{ \{x_j, \omega_j\} \in \mathcal{D}_\Gamma}
      \varphi(x_j)\omega_j \Bigr]^2. \label{eq:opt2-interface}
    \end{align}
    Note that the last term is added to make the solution unique.

    \item Find the solution $(u^c_1, u^c_2)\in \Phi_{N N} (d, 2)$ by minimizing:
    \begin{align}
      \mathcal{L}_3(u_c)
      :=
      & \sum_{i=1,2}\sum_{ \{ x_j, \omega_j \} \in
      \mathcal{D}_i} \bigl| \mathcal{A} \nabla (u^c_i (x_j) - u_1 (x_j) ) + \mathcal{A}^{-1} \curl
      \varphi (x_j) \bigr|^2\omega_j
      \nonumber\\
      & + \Bigl[\sum_{ \{ x_j, \omega_j \} \in
      \mathcal{D}_\Gamma} ( u^c_2(x_j) - g(x_j) ) \omega_j \Bigr]^2 \nonumber\\
      & + \Bigl[\sum_{ \{ x_j, \omega_j \} \in
      \mathcal{D}_{\Gamma_0}} ( u^c_1(x_j) - u^c_2(x_j) - \kappa_1(x_j) ) \omega_j \Bigr]^2.
      \label{eq:opt3-interface}
    \end{align}
    The last term is a regularization term to make the integration of $u_c$
    and $g$ on boundary $\partial\Omega$ are equal to each other, such that $u_c$ is a
    proper numerical approximation of $u$.
  \end{enumerate}

  Again, we will optimize $\mathcal{L}_1 + \mathcal{L}_2 +
  \mathcal{L}_3$ all in once.

\section{Numerical experiments}
\label{sec:ne}

We take $\Omega = [-1, 1]^2$, and test following examples with given exact
solutions.
\begin{itemize}
  \item {{Example 1}}: A Poisson Dirichlet boundary value problem
  \eqref{eq:PoissonStrong}  with the exact solution
  \begin{equation}\label{eq:ex1}
      u(x) = x_1^2+x_2^2 + sin(x+y)
    \end{equation}

  \item {{Example 2}}: An elliptic equation with smooth variable coefficients and
  Dirichlet boundary condition \eqref{eq:varc-strongform}.
  The variable coefficient matrix is taken as
  \begin{equation}
    \label{eq:ex2coefA}
    \mathcal{A} = \begin{pmatrix}
      1+x_1^2, & 0 \\
      0, & 1
    \end{pmatrix}.
  \end{equation}
  The exact solution is given as
  \begin{equation}\label{eq:ex2}
      u(x) = \displaystyle e^{\cos \bigl( x_1 + x_2^2 \bigr)} .
    \end{equation}

    \item {{Example 3}}: An elliptic equation with non-smooth variable coefficients and
  Dirichlet boundary condition \eqref{eq:varc-strongform}.
  The variable coefficient matrix is taken as
  \begin{equation}
    \label{eq:ex3coefA}
    \mathcal{A} = \begin{pmatrix}
      1+x_1^2, & 0 \\
      0, & 1+|x_2|
    \end{pmatrix}.
  \end{equation}
  The exact solution is taken as
  \begin{equation}\label{eq:ex3}
      u(x) = \displaystyle e^{\cos \bigl( x_1 + |x_2|^3 \bigr)} .
    \end{equation}

    \item {Example 4}: An elliptic equation
    \eqref{eq:varc-strongform} with discontinuous
    coefficients.
  The coefficient matrix is taken as
  \begin{equation}
    \label{eq:ex4coefA}
    \mathcal{A} = \begin{pmatrix}
        1, & 0 \\
        0, & \frac43-\frac23\text{sgn}(x_2)
      \end{pmatrix}.
  \end{equation}
  The exact solution is given as
  \begin{equation}\label{eq:ex4}
      u(x) = e^{\cos(x_1+x_2)} + \frac12 e^{\cos(x_1+|x_2|)}.
    \end{equation}

    \item {{Example 5}}: An interface problem
    \eqref{eq:InterfaceS1}-\eqref{eq:InterfaceS3}. The domain $\Omega_1=[-0.5,
    0.5]^2$, $\Omega_2 = \Omega\backslash\Omega_1$.
  The coefficient matrices are taken as
  \begin{equation}
    \label{eq:ex5coefA}
    \mathcal{A}|_{\Omega_1} = \begin{pmatrix}
      10, & 0 \\
      0, & 10
    \end{pmatrix},
    \quad
    \mathcal{A}|_{\Omega_2} = \begin{pmatrix}
        1+x_1^2, & 0 \\
        0, & 1
      \end{pmatrix}.
  \end{equation}
  The exact solution is given as
  \begin{equation}\label{eq:ex5}
      u(x)|_{\Omega_1} = 5 \displaystyle e^{- ( x_1^2 + x_2^2
      )} ,
      \quad
      u(x)|_{\Omega_2} = 4\displaystyle e^{\cos \bigl( x_1^2/2 + x_2^2
      \bigr)-1}.
    \end{equation}

\end{itemize}

The implementation is conducted using PyTorch \cite{paszke_pytorch_2019}. We employ ResNets with five ResBlock layers, where each ResBlock is a shallow network comprising 20, 35, and 35 hidden units for the proposed method, the standard deep Ritz method, and the PINN method, respectively, ensuring comparable parameter counts across these methods. We test activation functions ReCUr, Tanh, and ReQUr, with further details on ResNet structures and RePUr activations available in \cite{he_deep_2016, li_better_2020, yu_onsagernet_2021a, tian_deep_2024}.

For training data, we generate 10,000 inner and boundary points using composite Gaussian quadrature of order 5. An additional 10,000 points on a uniform grid serve as test data. The Adam optimizer is initially used to train the model for 100 epochs, with a batch size of 200 and an initial learning rate of 0.005. A learning rate scheduler, CosineAnnealingLR \cite{loshchilov_sgdr_2017}, is employed to adjust the learning rate. Following the Adam training phase, we further refine the model using the L-BFGS optimizer for 50 steps, with a history size of 100 and 60 inner iterations per step.

Results from both the proposed method and the standard deep Ritz method (using a boundary penalty constant $\beta=1000$) for these examples are presented in the following figures.
We have several observations:
\begin{itemize}
    \item By comparing \ref{fig:ex1-ndb} and \ref{fig:ex1-drm}, we observe that in the deep Ritz method, the boundary condition is quickly learned due to the large penalty constant $\beta=1000$. However, while this high penalty facilitates rapid boundary condition learning, it significantly slows the convergence of the solution within the domain, particularly for the Tanh network. We attribute this to the increased optimization complexity introduced by the large penalty. In contrast, the new penalty-free method learns both boundary values and the entire function inside the computational domain simultaneously, achieving a more accurate numerical solution with the same number of training points and training steps.
    \item The results for the two variable-coefficient cases, Examples 2 and 3, are comparable. The non-smoothness in the diffusion coefficients does not introduce significant additional difficulty, as the numerical error in the non-smooth coefficient case is only slightly higher than in the smooth coefficient case. Additionally, we observe that ReCUr DNNs outperform those using the other two activation functions, with RePUr DNNs demonstrating faster feature learning than Tanh networks, particularly during the Adam training phase.
    \item During training with Adam, ReQUr networks typically perform well; however, their performance occasionally degrades during the LBFGS training steps, particularly when compared to ReCUr, RePU4r, and Tanh networks. This may be due to the LBFGS method’s reliance on high-order information, and the higher-order derivatives of ReQUr networks are less smooth than those of the other networks.
    
    In the PINN results for Example 4 (Fig. \ref{fig:ex4-pinn}), we observe that PINN performs poorly in this case compared to the proposed natural deep Ritz method. This discrepancy arises because the exact solution is not a strong solution, while the PINN method depends on the strong form of the equation.
    \item For the interface problem in Example 5, the results in Fig. \ref{fig:ex5train-ndb} and \ref{fig:ex5-ndb} indicate that the proposed natural deep Ritz method can efficiently solve the interface problem without requiring penalty parameter tuning for boundary conditions and interface jump conditions. The relative $L^2$ error is approximately $2.5 \times 10^{-3}$, and the relative $L^\infty$ error for this test is around $5 \times 10^{-3}$. However, the LBFGS optimizer is less efficient in this case than in previous examples. Developing more efficient training methods for interface problems warrants further investigation.
\end{itemize}

\begin{figure}
	\centering
    \includegraphics[width=0.8\textwidth]{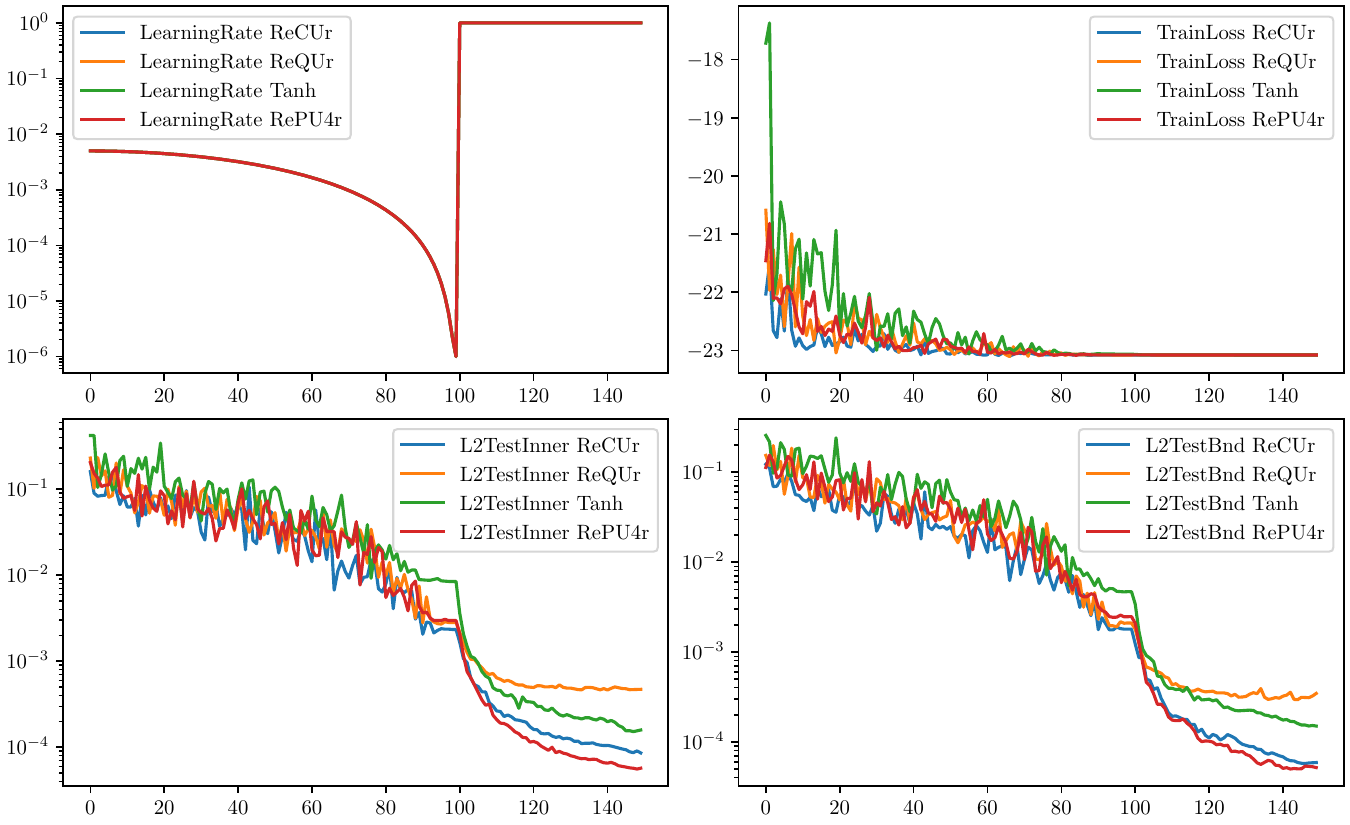}
    \caption{\label{fig:ex1train-ndb} Training results :
    The learning rate (top-left), training loss (top-right)
    $L2$ Testing error on $\Omega$ (bottom-left) and boundary $\Gamma$
    (bottom-right) for Example 1 using New method}
\end{figure}
\begin{figure}
	\centering
    \includegraphics[width=0.8\textwidth]{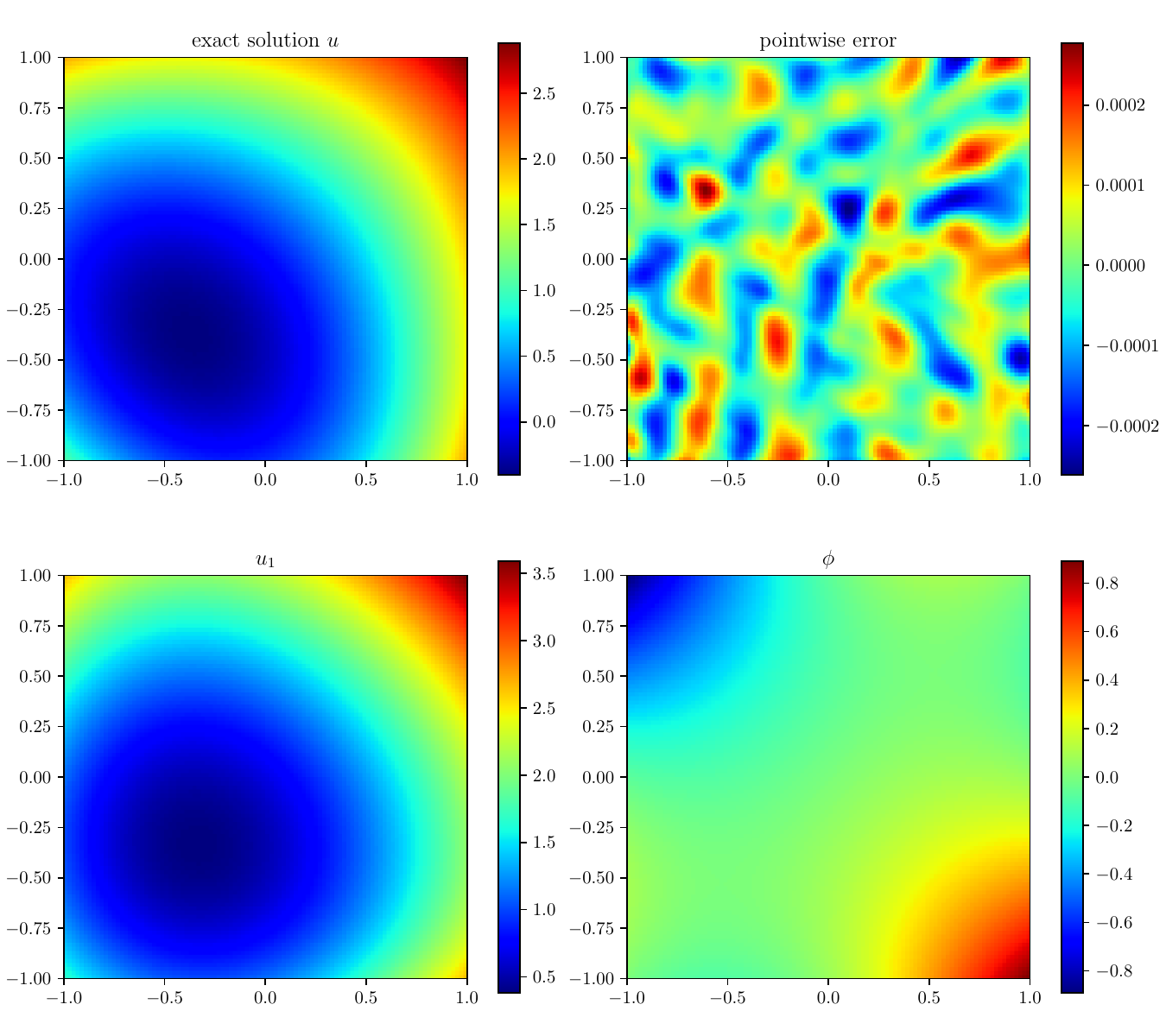}
    \caption{\label{fig:ex1-ndb} The exact solution and
    learned solution for Example 1 using New method}
\end{figure}

\begin{figure}
	\centering
    \includegraphics[width=0.8\textwidth]{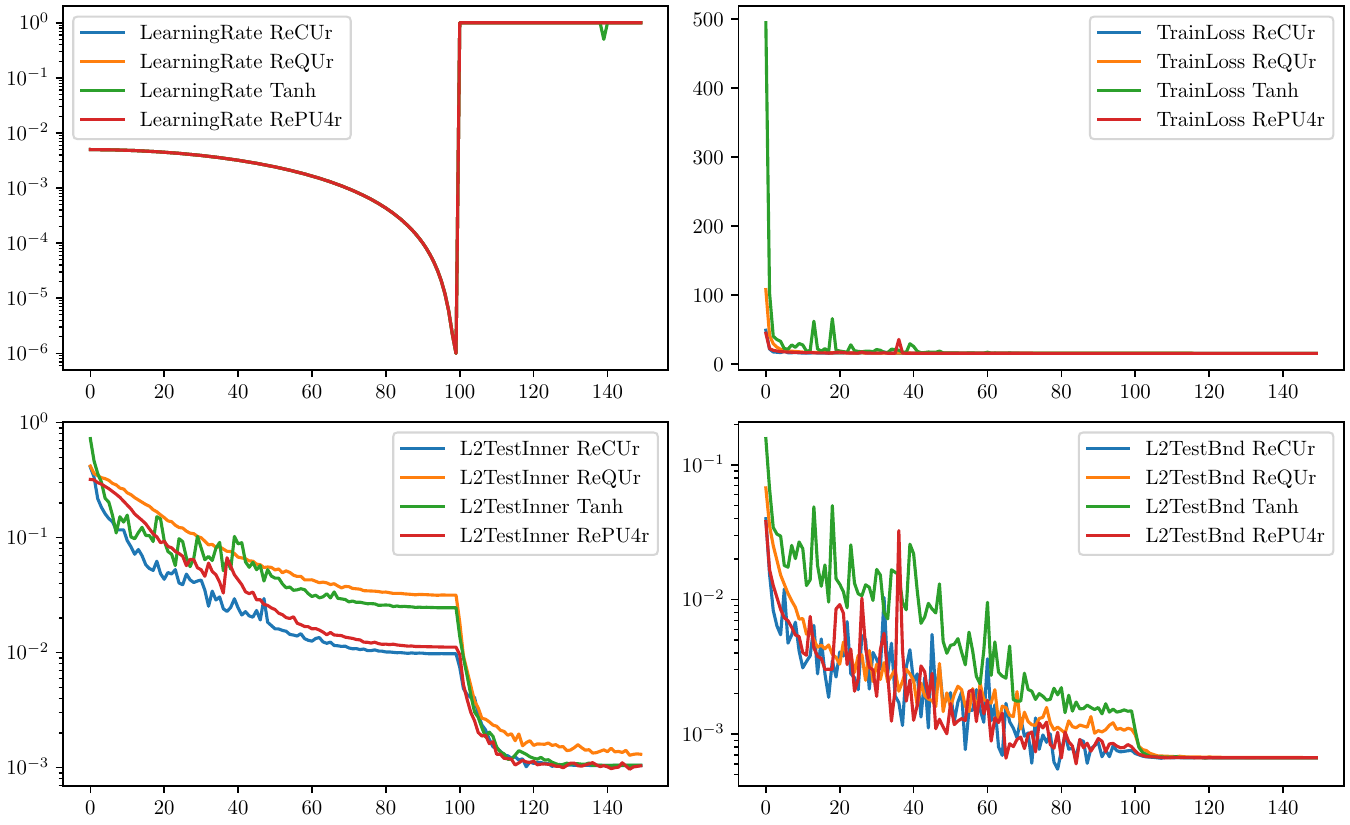}
    \caption{\label{fig:ex1train-drm} Training results :
    The learning rate (top-left), training loss (top-right)
    $L2$ Testing error on $\Omega$ (bottom-left) and boundary $\Gamma$
    (bottom-right) for Example 1 using Deep Ritz method}
\end{figure}
\begin{figure}
	\centering
    \includegraphics[width=0.8\textwidth]{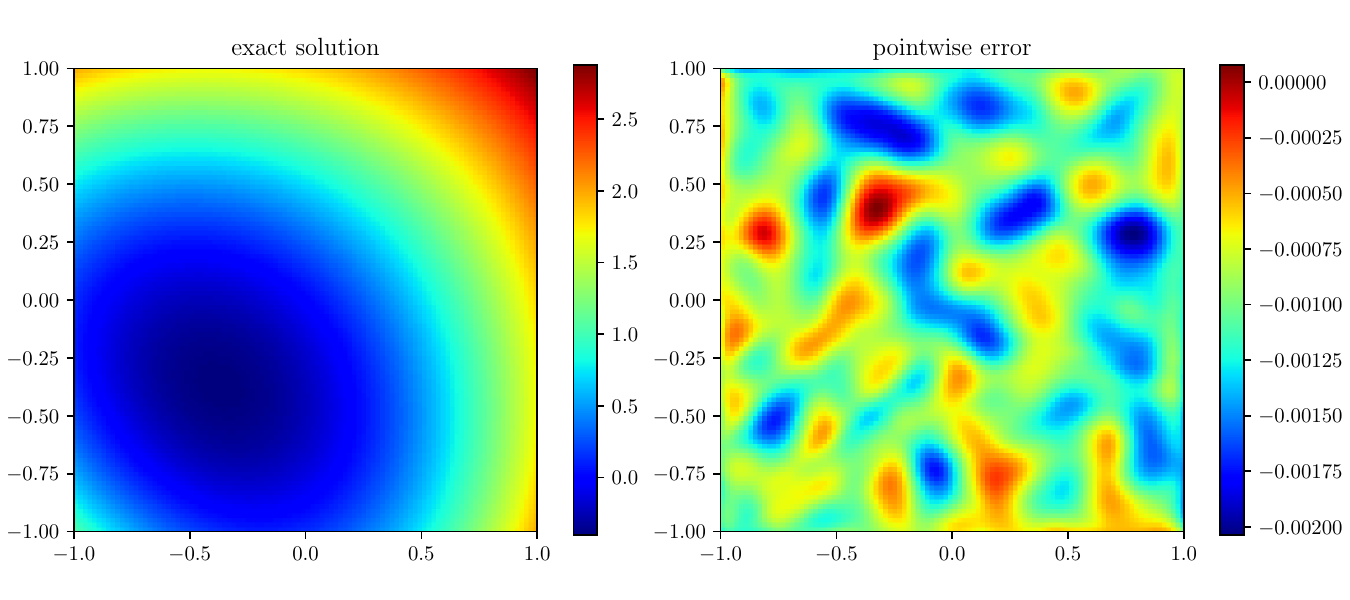}
    \caption{\label{fig:ex1-drm} The exact solution and
    learned solution for for Example 1 using Deep Ritz method}
\end{figure}

\begin{figure}
	\centering
    \includegraphics[width=0.8\textwidth]{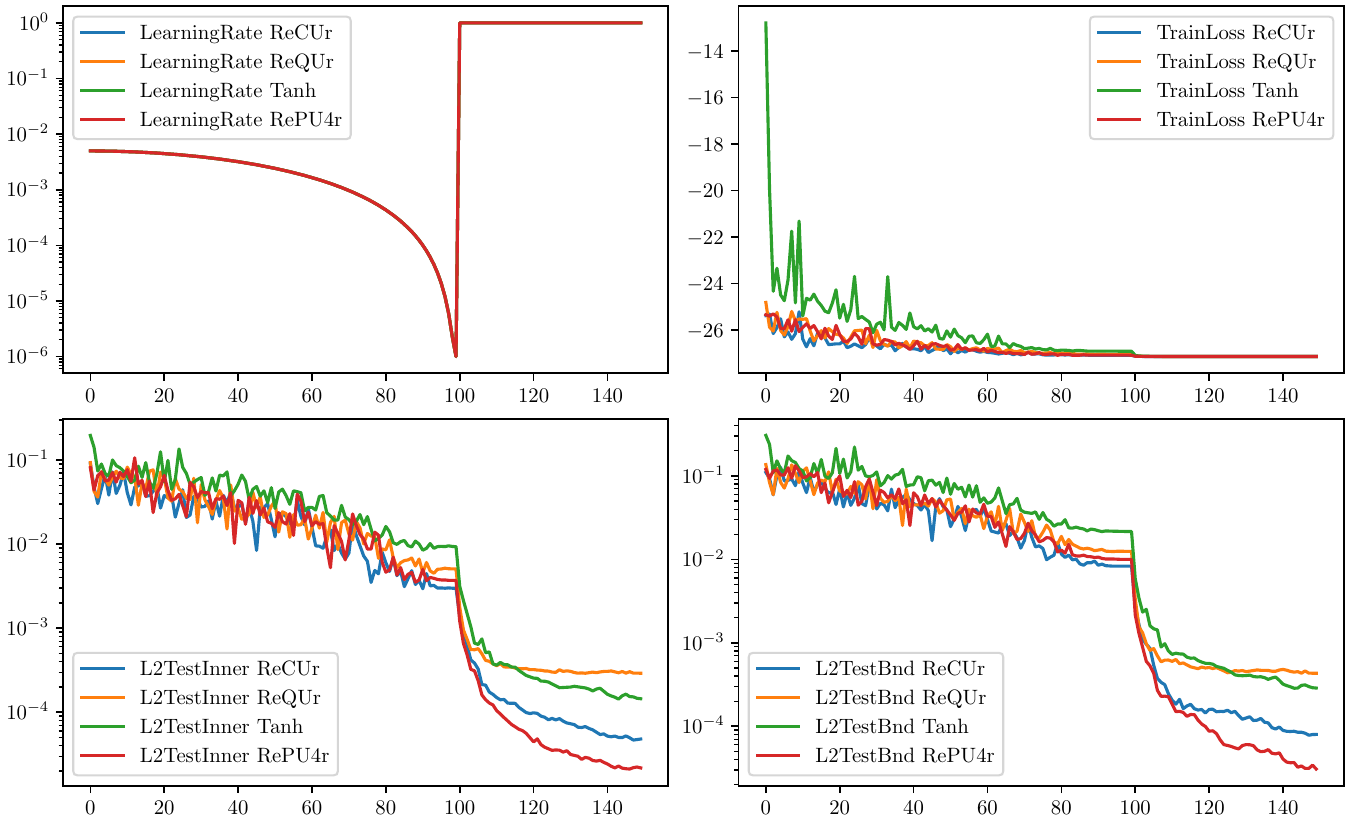}
    \caption{\label{fig:ex2train-ndb} Training results :
    The learning rate (top-left), training loss (top-right)
    $L2$ Testing error on $\Omega$ (bottom-left) and boundary $\Gamma$
    (bottom-right) for Example 2 using New method}
\end{figure}
\begin{figure}
	\centering
    \includegraphics[width=0.8\textwidth]{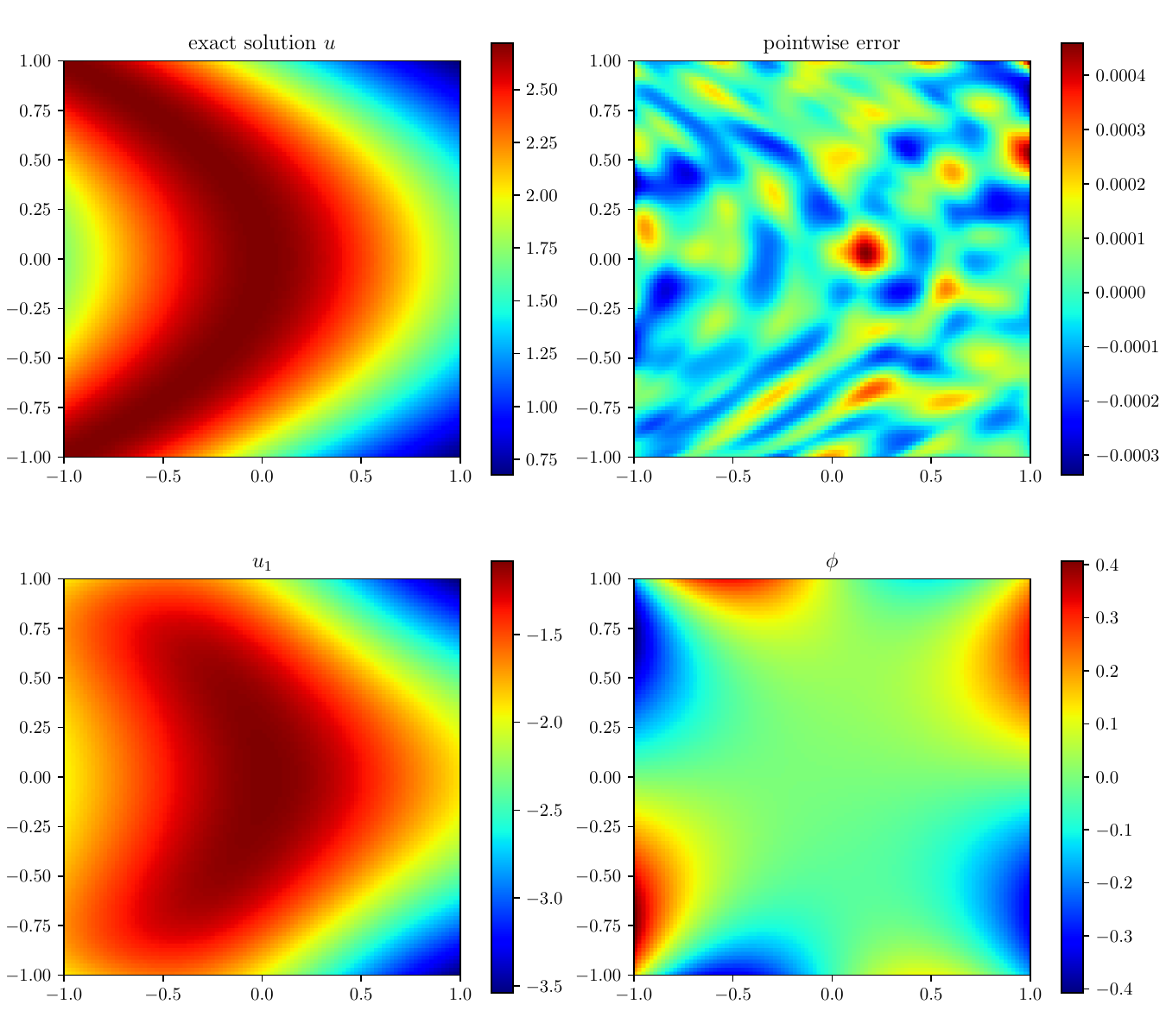}
    \caption{\label{fig:ex2-ndb} The exact solution and
    learned solution using RePUr neural networks for Example 2 using New method}
\end{figure}

\begin{figure}
	\centering
    \includegraphics[width=0.8\textwidth]{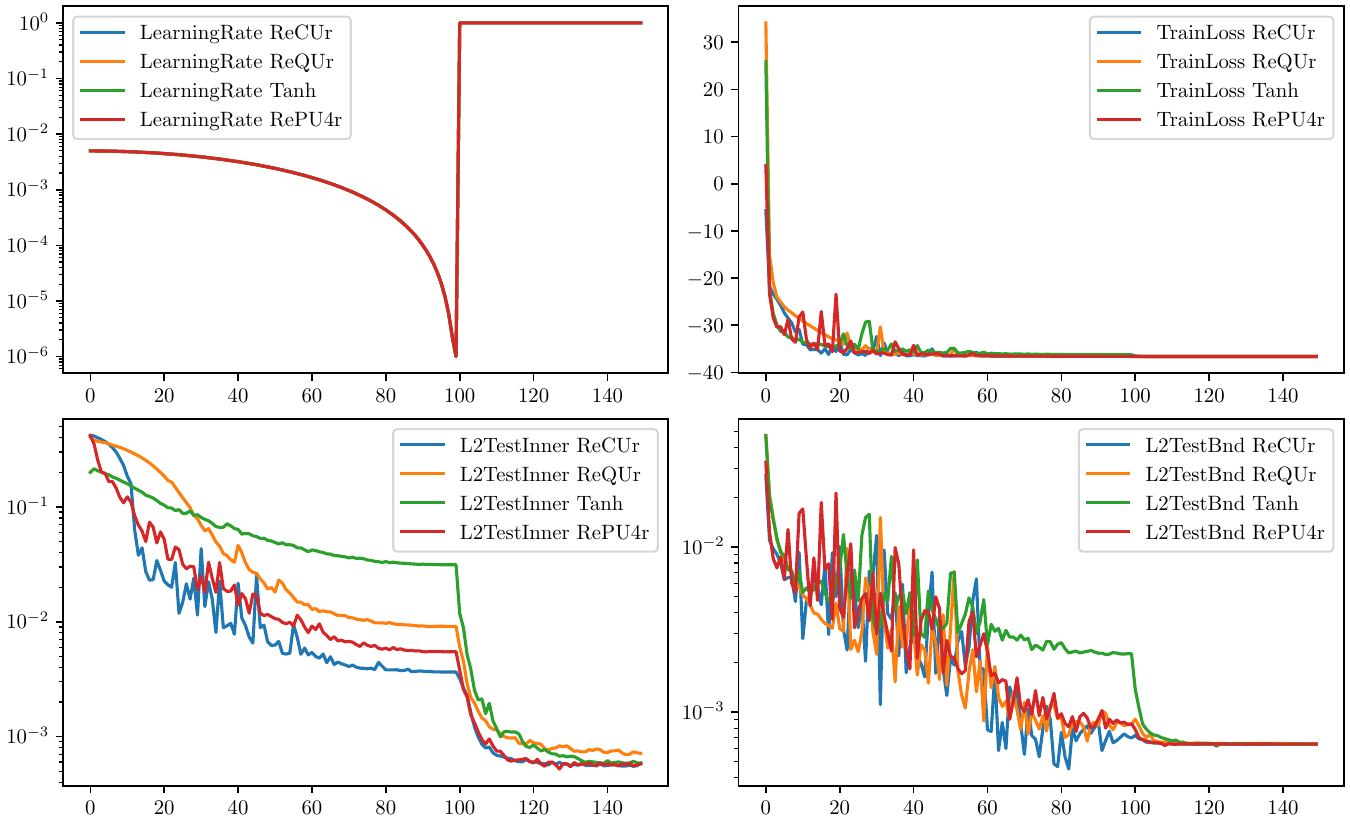}
    \caption{\label{fig:ex2train-drm} Training results :
    The learning rate (top-left), training loss (top-right)
    $L2$ Testing error on $\Omega$ (bottom-left) and boundary $\Gamma$
    (bottom-right) for Example 2 using Deep Ritz method}
\end{figure}
\begin{figure}
	\centering
    \includegraphics[width=0.8\textwidth]{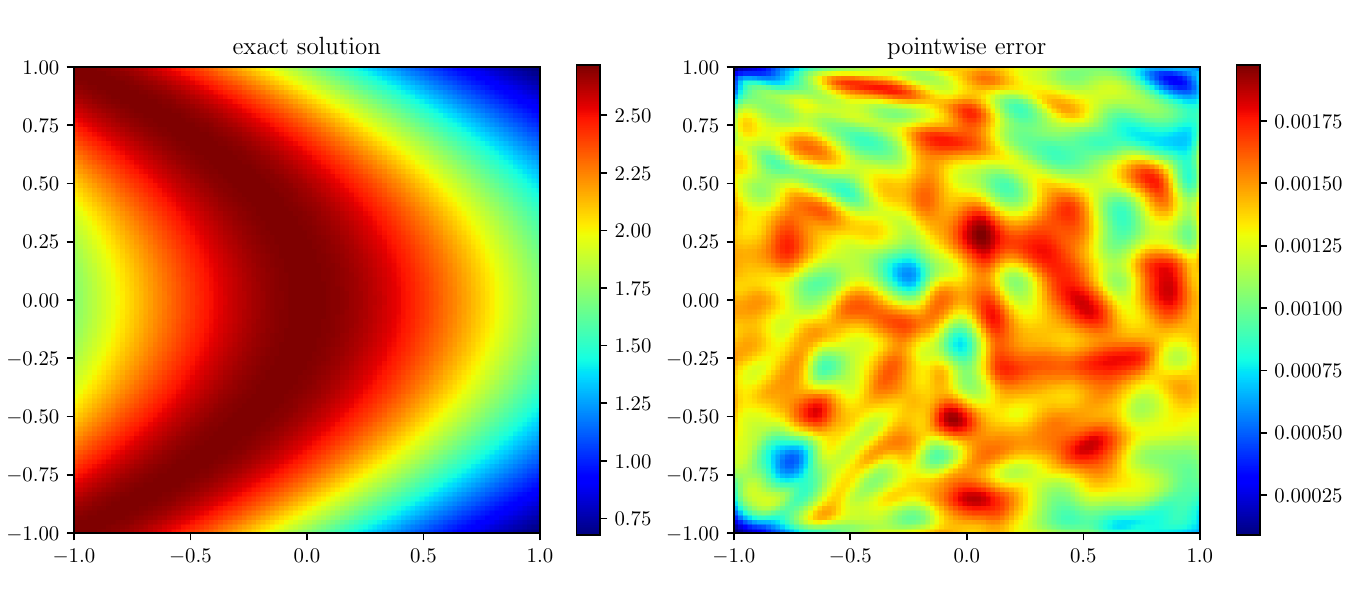}
    \caption{\label{fig:ex2-drm} The exact solution and
    learned solution using RePUr neural networks for Example 2 using Deep Ritz method}
\end{figure}

\begin{figure}
	\centering
    \includegraphics[width=0.8\textwidth]{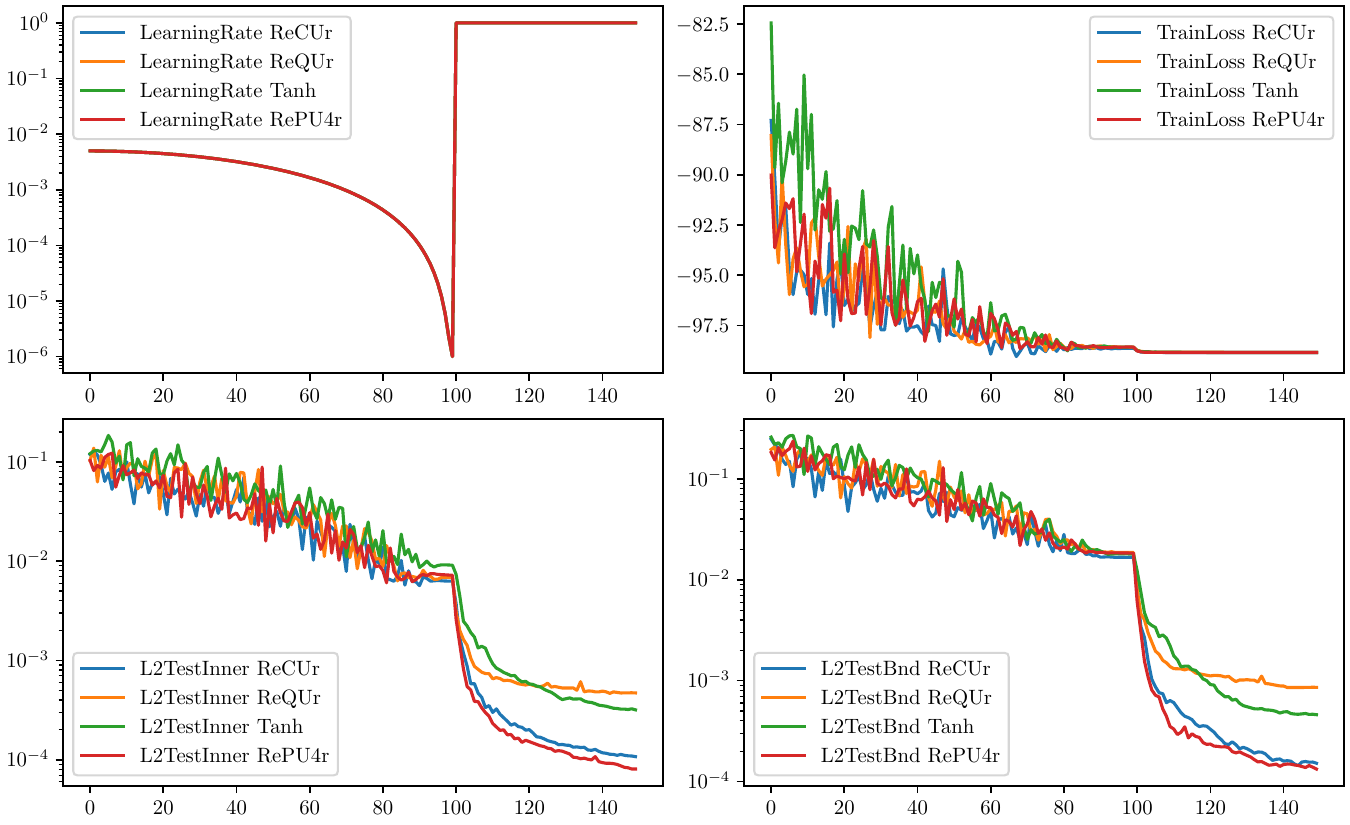}
    \caption{\label{fig:ex3train-ndb} Training results :
    The learning rate (top-left), training loss (top-right)
    $L2$ Testing error on $\Omega$ (bottom-left) and boundary $\Gamma$
    (bottom-right) for Example 3 using New method}
\end{figure}
\begin{figure}
	\centering
    \includegraphics[width=0.8\textwidth]{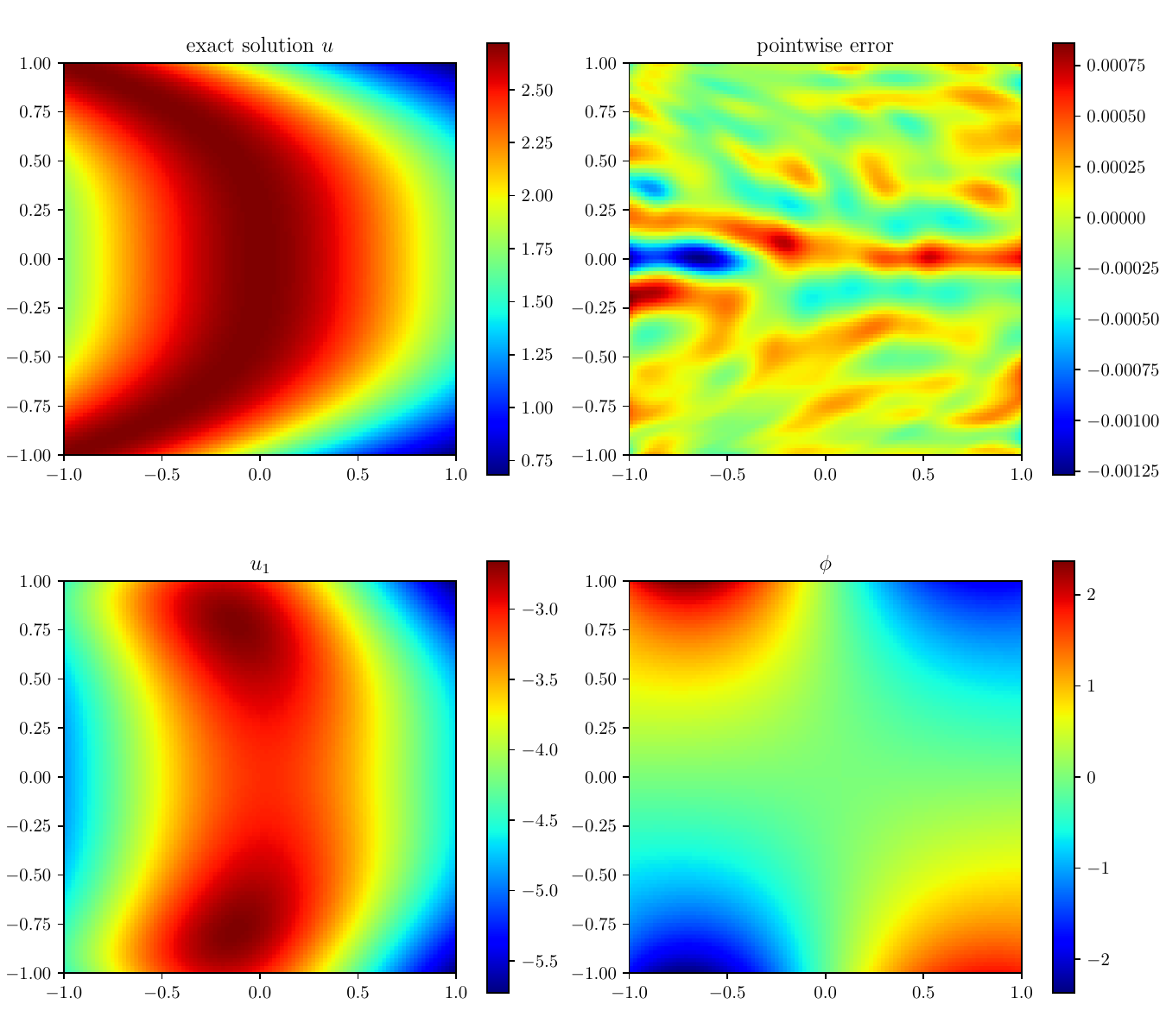}
    \caption{\label{fig:ex3-ndb} The exact solution and
    learned solution using RePUr neural networks for Example 3 using New method}
\end{figure}

\begin{figure}
	\centering
    \includegraphics[width=0.8\textwidth]{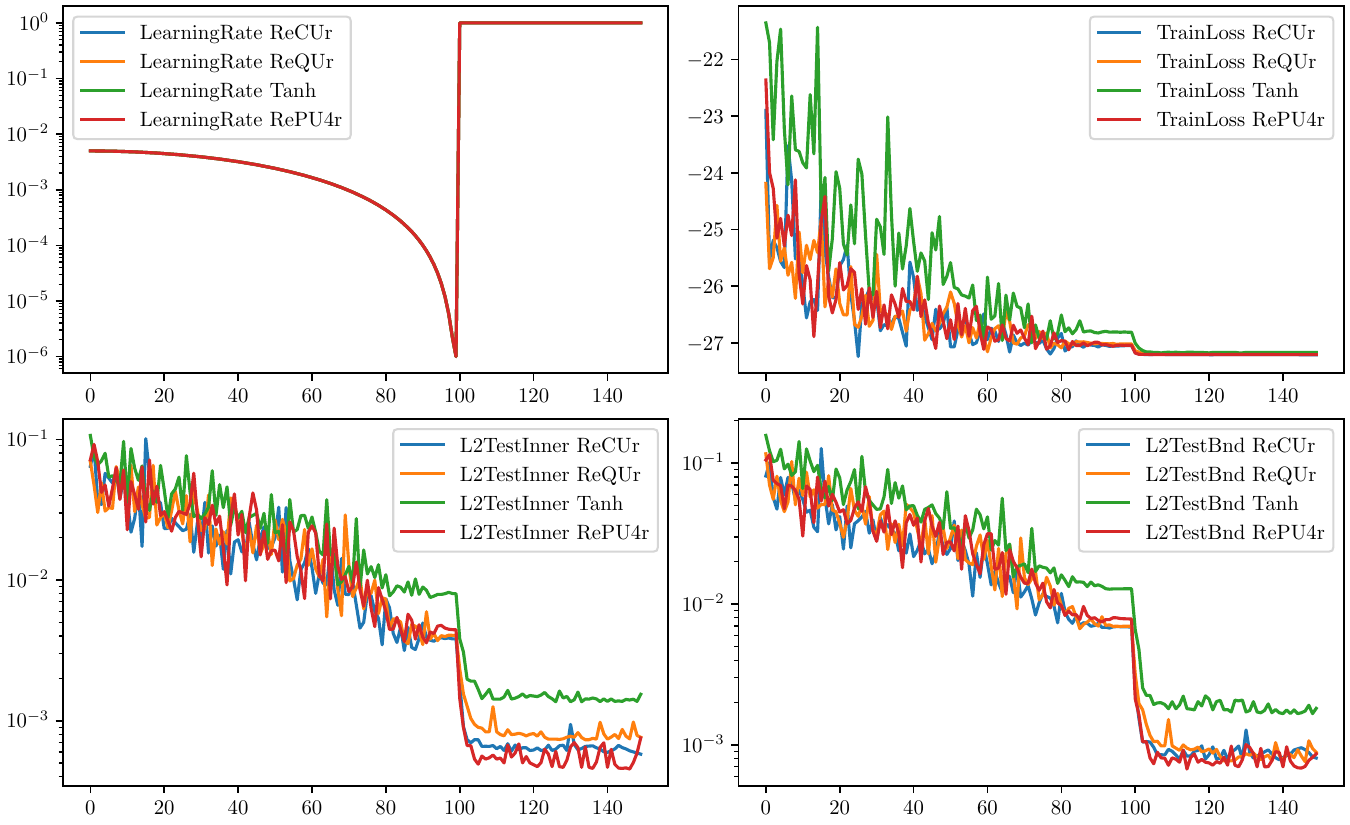}
    \caption{\label{fig:ex4train-ndb} Training results :
    The learning rate (top-left), training loss (top-right)
    $L2$ Testing error on $\Omega$ (bottom-left) and boundary $\Gamma$
    (bottom-right) for Example 4 using New method}
\end{figure}
\begin{figure}
	\centering
    \includegraphics[width=0.8\textwidth]{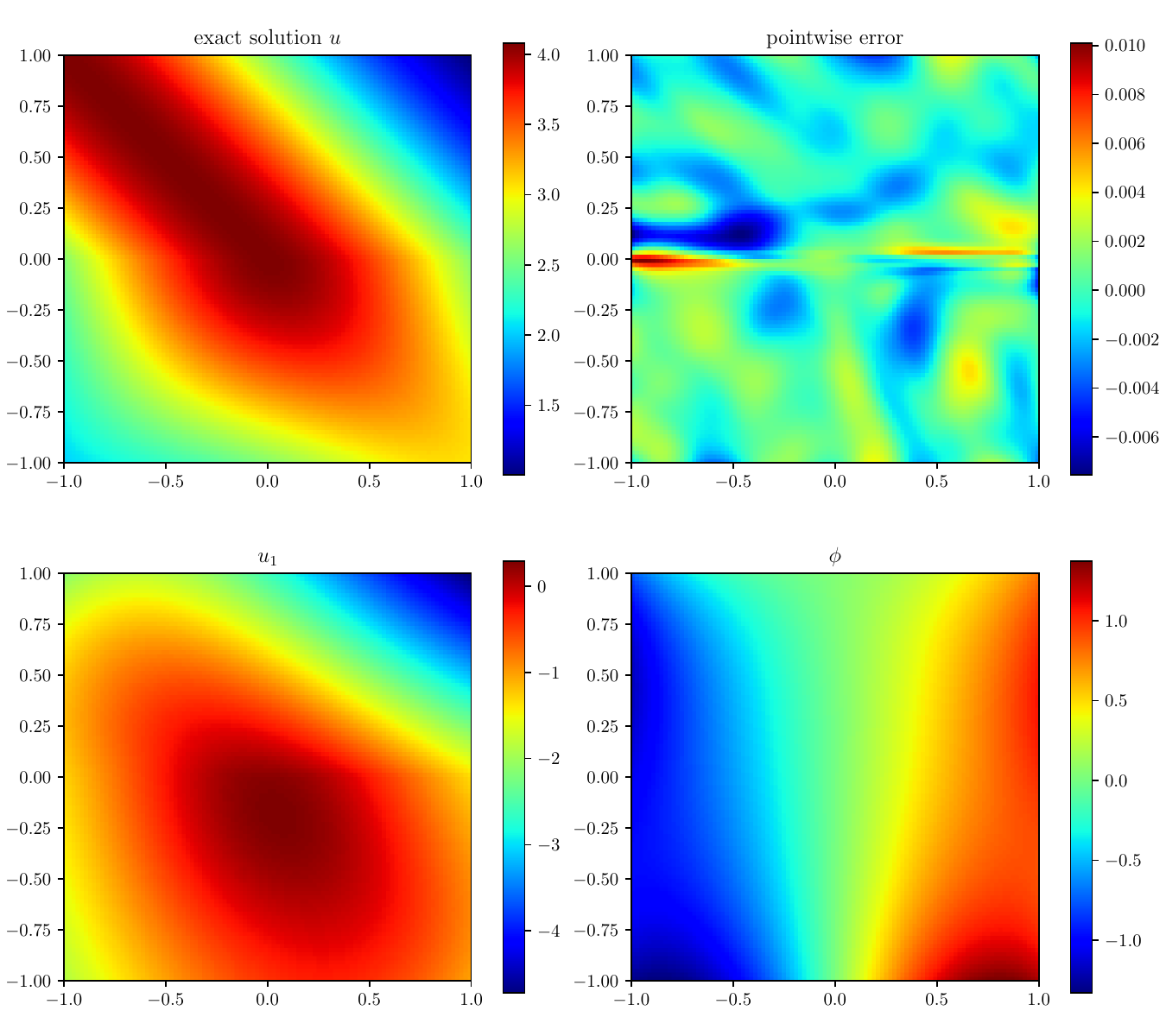}
    \caption{\label{fig:ex4-ndb} The exact solution and
    learned solution using RePUr neural networks for Example 4 using New method}
\end{figure}
\begin{figure}
	\centering
    \includegraphics[width=0.8\textwidth]{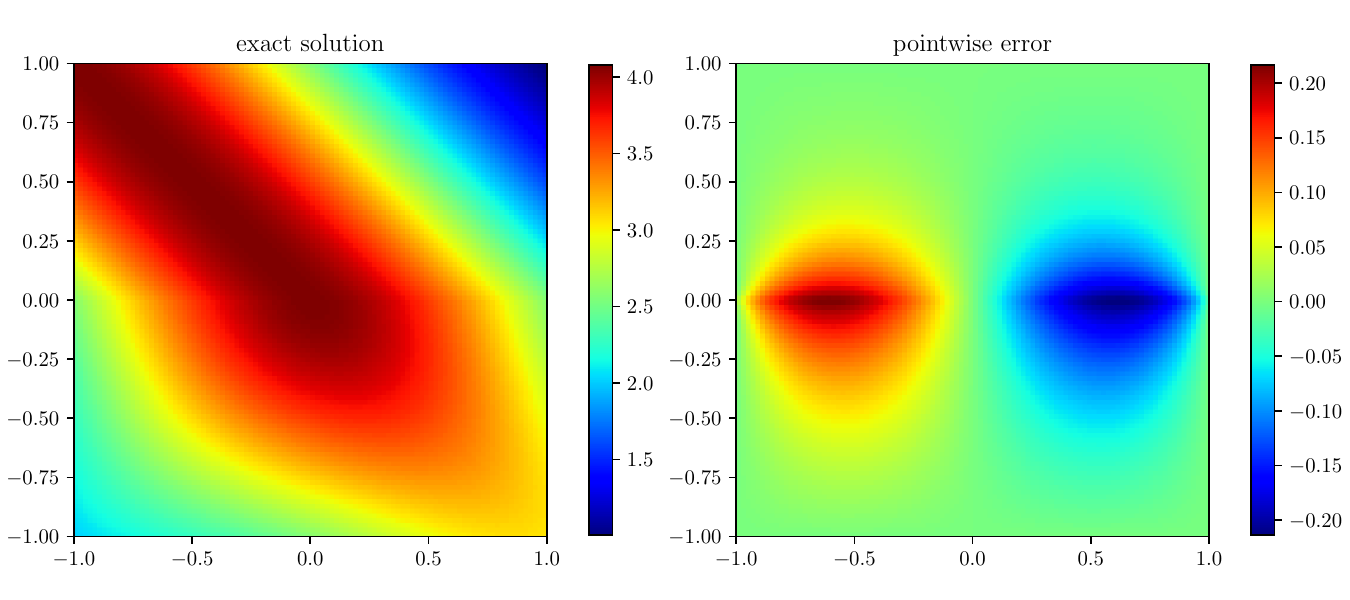}
    \caption{\label{fig:ex4-pinn} The exact solution and
    learned solution for Example 4 using PINN.}
\end{figure}

\begin{figure}
	\centering
    \includegraphics[width=0.8\textwidth]{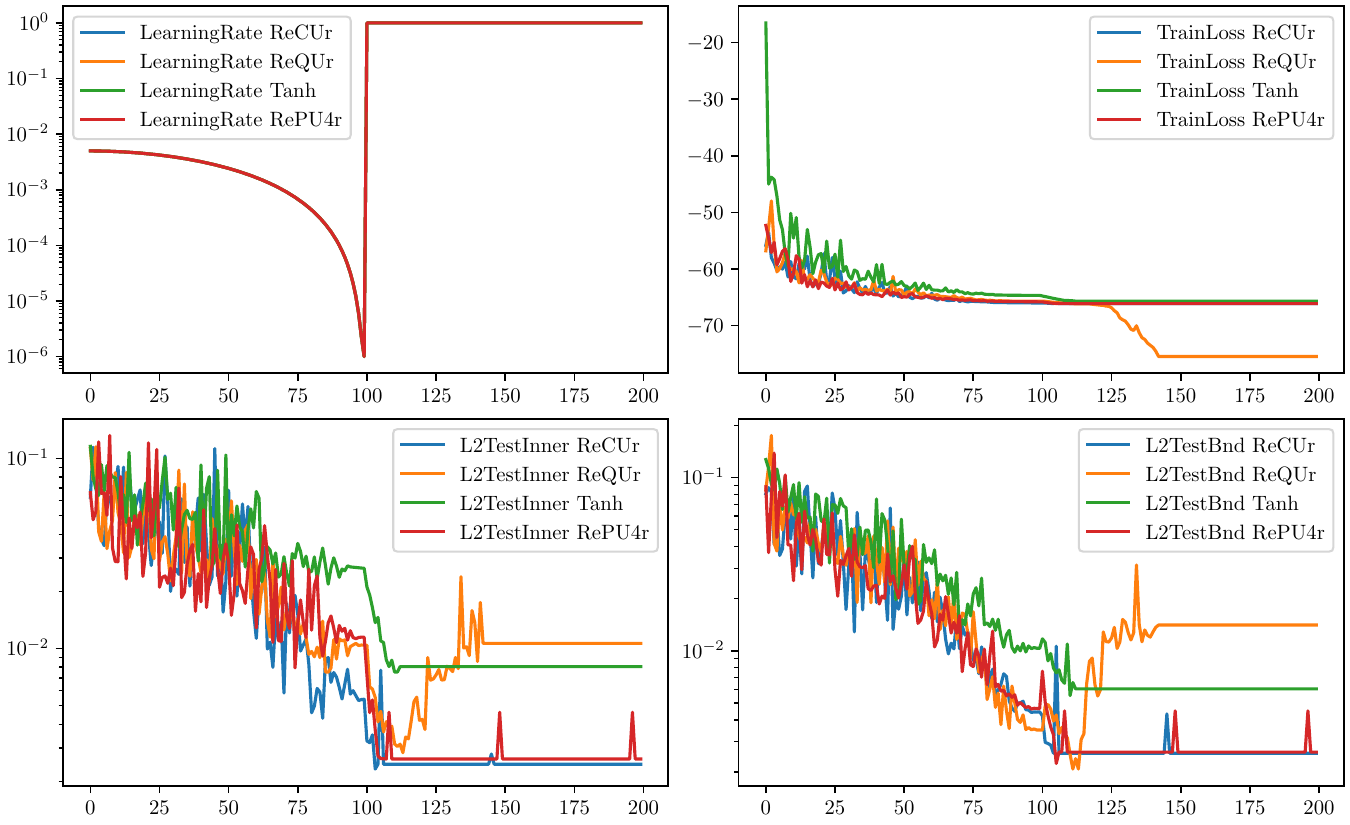}
    \caption{\label{fig:ex5train-ndb} Training results :
    The learning rate (top-left), training loss (top-right)
    $L2$ Testing error on $\Omega$ (bottom-left) and boundary $\Gamma$
    (bottom-right) for Example 5 using New method}
\end{figure}
\begin{figure}
	\centering
    \includegraphics[width=0.8\textwidth]{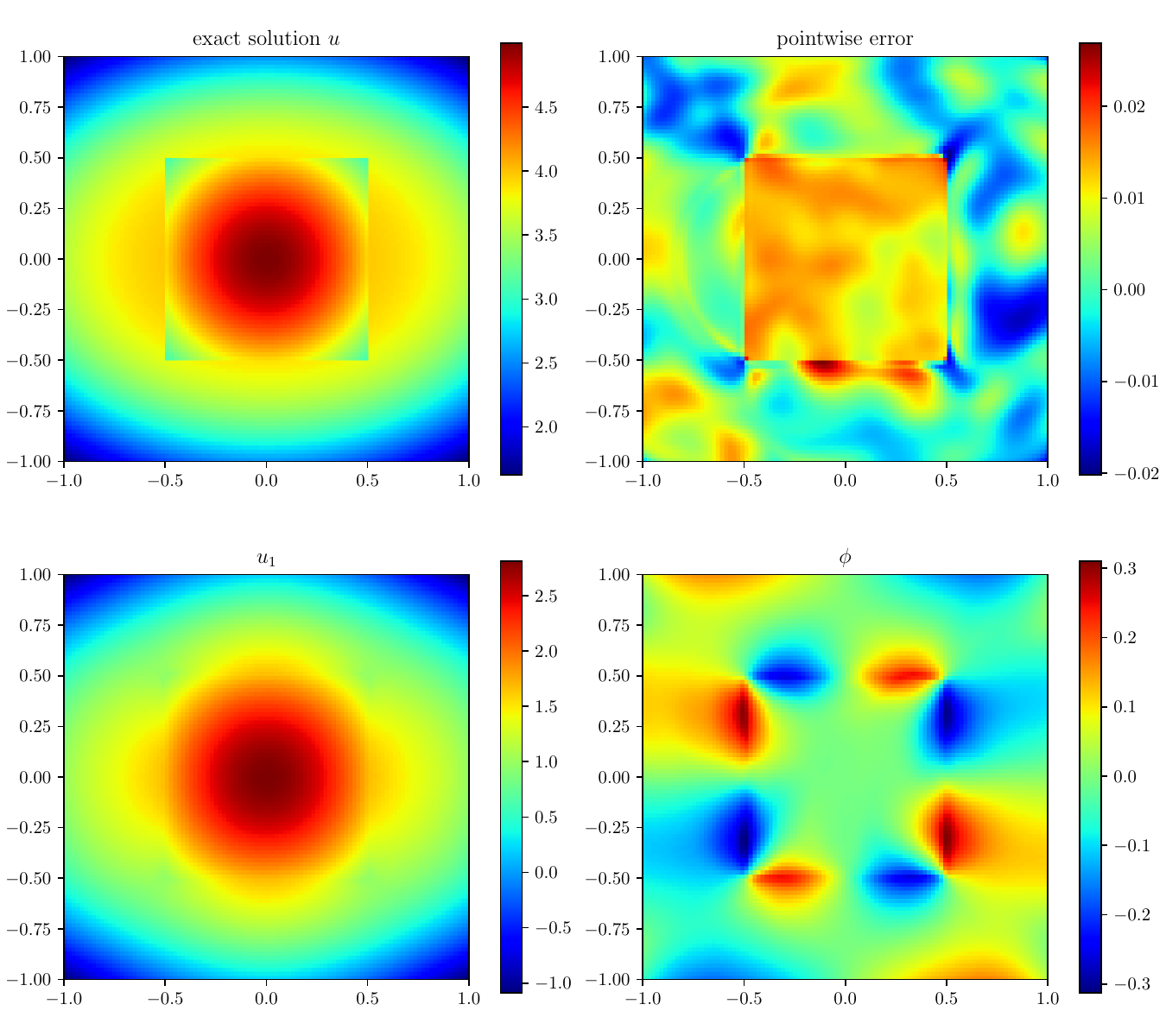}
    \caption{\label{fig:ex5-ndb} The exact solution and
    learned solution using RePUr neural networks for Example 4 using New method}
\end{figure}

\section{Concluding remarks}
\label{sec:conc}

In this paper, we develop a novel strategy for solving essential boundary value problems using neural networks by transforming the original problem into a series of pure natural boundary value problems, which can then be effectively solved using the deep Ritz method. Various model problems are employed to demonstrate the advantages of this approach. While this study focuses on two-dimensional elliptic problems with essential boundary conditions only, several potential extensions of the proposed method are possible:
\begin{enumerate}
\item Extension of the approach, readily, to other types of classical meshfree methods besides neural networks; 
\item Extension to other boundary conditions, e.g. mixed type condition;  
\item Extension to other self-adjoint problems provided relevant complex dualities, and the extension to non-self-adjoint problems is possible; 
\item Extension, dedicatedly, to three-dimensional and higher dimensional problems. 
\end{enumerate}
We will present some of these extensions in future works.

\section*{Acknowledgments}

This work was partially supported by the National Natural Science Foundation of
China (grant numbers 92370205, 12171467, 12271512, 12161141017). The
computations were partially done on the high-performance computers of the State
Key Laboratory of Scientific and Engineering Computing, Chinese Academy of
Sciences.

%
%
%


\end{document}